\documentclass[10pt,a4paper,reqno]{amsart}
\usepackage[utf8]{inputenc}
\usepackage[T1]{fontenc}       
\usepackage{ae}   
\usepackage{amsfonts}         
\usepackage{amsmath}
\usepackage{amsthm}
\usepackage{amssymb}
\usepackage{mathtools}
\usepackage{calc}
\usepackage{centernot}
\usepackage{color}
\usepackage[pdftex]{hyperref}
\usepackage{mathrsfs}

\usepackage{xargs}

\usepackage{todonotes}

\newcommandx{\huom}[2][1=]{\todo[linecolor=red,backgroundcolor=red!10,bordercolor=red,#1]{#2}}

\newcommand{\define}[1]{{\em #1\/}}

\newcommand{\ve}{\varepsilon}
\newcommand{\N}{\mathbb{N}}

\newcommand{\R}{\mathbb{R}}

\newcommand{\abs}[1]{\lvert #1 \rvert}
\newcommand{\norm}[1]{\lVert #1 \rVert}

\newcommand{\ang}[1]{\langle #1 \rangle}

\newcommand{\p}{\partial}

\DeclareMathOperator\diam{diam}
\DeclareMathOperator\dist{dist}
\DeclareMathOperator\dv{div}

\newcommand{\pinf}{\partial_{\infty}}

\newcommand{\nb}{\bar{\nabla}}
\newcommand{\n}{\nabla}

\DeclareMathOperator\Ric{Ric}
\newcommand{\tr}{\text{tr }}

\DeclareMathOperator\Hess{Hess}
\DeclareMathOperator\vol{vol}
\DeclareMathOperator\spt{spt}

\numberwithin{equation}{section}

\theoremstyle{plain}
\newtheorem{thm}{Theorem}[section]
\newtheorem{lem}[thm]{Lemma}
\newtheorem{cor}[thm]{Corollary}

\theoremstyle{definition}

\newtheorem{exa}[thm]{Example}

\begin{document}

\title{Dirichlet problem for $f$-minimal graphs}

\author{Jean-Baptiste Casteras}
\address{J.-B. Casteras, Departement de Mathematique
Universite libre de Bruxelles, CP 214, Boulevard du Triomphe, B-1050 Bruxelles, Belgium}
\email{jeanbaptiste.casteras@gmail.com}

\author{Esko Heinonen}
\address{E. Heinonen, Department of Mathematics and Statistics, P.O.B. 68 (Gustaf 
H\"allstr\"omin katu 2b), 00014 University
of Helsinki, Finland.}
\email{ea.heinonen@gmail.com}

\author{Ilkka Holopainen}
\address{I.Holopainen, Department of Mathematics and Statistics, P.O.B. 68 (Gustaf 
H\"allstr\"omin katu 2b), 00014 University
of Helsinki, Finland.}
\email{ilkka.holopainen@helsinki.fi}

\thanks{J.-B.C. supported by MIS F.4508.14 (FNRS)}
\thanks{E.H. supported by Jenny and Antti Wihuri Foundation.}



\begin{abstract}
We study the asymptotic Dirichlet problem for $f$-minimal graphs in Cartan-Hadamard manifolds $M$. $f$-minimal hypersurfaces are natural generalizations of 
self-shrinkers which play a crucial role in the study of mean curvature flow. In the first part of this paper, we prove the existence of $f$-minimal 
graphs with prescribed boundary behavior on a bounded domain $\Omega \subset M$ under suitable assumptions on $f$ and the boundary of $\Omega$. In the second part, we consider the asymptotic Dirichlet problem. Provided that $f$ decays fast enough, we construct solutions to the problem. Our assumption on the decay of $f$ is linked with the sectional curvatures of $M$. In view of a result of Pigola, Rigoli and Setti, our results are almost sharp. 
\end{abstract}

\maketitle

\tableofcontents

\section{Introduction}
In this paper we study the Dirichlet problem for the so-called $f$-minimal graph equation on a complete non-compact 
$n$-dimensional Riemannian manifold $M$ with the Riemannian metric given by $ds^2 = \sigma_{ij}dx^i dx^j$ in local coordinates. 
We equip $N=M\times\R$ with the product metric $ds^2 +dt^2$ and assume that $f\colon N\to\R$ is a smooth function. The Dirichlet
problem for $f$-minimal graphs is to find a solution $u$ to the equation
\begin{equation} \label{mingrapheq}
 \begin{cases}
  \dv \dfrac{\nabla u}{\sqrt{1+|\nabla u|^2}} = \ang{\nb f,\nu} \quad \text{in } \Omega \\
  u|\p\Omega = \varphi,
 \end{cases}
\end{equation}
where $\Omega \subset M$ is a bounded domain, $\bar{\n}f$ is the gradient of $f$ with respect to the product Riemannian metric, and $\nu$ denotes the downward unit normal to the graph of $u$, i.e.
 \begin{equation}\label{unitnormal}
     \nu = \frac{(\nabla u,-1)}{\sqrt{1+|\nabla u|^2}}.
 \end{equation}
The regularity assumptions on $f,\ \p\Omega$, and on $\varphi$ will be specified in due course.
 
The equation \eqref{mingrapheq} can be written in non-divergence form as
 \begin{equation}
  \frac{1}{W} \left( \sigma^{ij} - \frac{u^i u^j}{W^2} \right) u_{i;j} = \ang{\nb f, \nu},
 \end{equation}
where $W= \sqrt{1+|\nabla u|^2}$, $(\sigma^{ij})$ stands for the inverse matrix of $(\sigma_{ij}),\ u^i = \sigma^{ij}u_j$,
with $u_j =\partial u/\partial x^j$,
and $u_{i;j}=u_{ij}-\Gamma^{k}_{ij}u_k$ denotes the second order covariant derivative of $u$.

We recall that an immersed hypersurface $\Sigma$ of a Riemannian manifold $(N,g)$ is called an $f$-minimal hypersurface if its (scalar)
mean curvature $H$ satisfies an equation
\[
H= \ang{\nb f,\nu}
\] 
at every point of $\Sigma$. Here, too, $\nu$ is a unit normal vector field along $\Sigma$, $f$ is a smooth function on $N$, and $\bar{\n}f$ denotes its 
gradient with respect to the Riemannian metric $g$. Hence the graph of a solution $u$ of \eqref{mingrapheq} is an $f$-minimal hypersurface in 
$M\times\R$. Note that we define the mean curvature as the trace of the second fundamental form. Other examples of $f$-minimal hypersurfaces are
\begin{itemize}
\item[(a)] minimal hypersurfaces if $f$ is identically constant,
\item[(b)] self-shrinkers in $\R^{n+1}$ if $f(x)=\abs{x}^2/4$,
\item[(c)] minimal hypersurfaces of weighted manifolds $M_f=\bigl(M,g,e^{-f}d\vol_{M}\bigr)$, where $(M,g)$ is a complete Riemannian manifold with the 
Riemannian volume element $d\vol_{M}$,
\item[(d)] translating solitons of the mean curvature flow in $M\times\R$ if $f(x,t) = -t$.
\end{itemize}
We refer to \cite{ColMin_CMH}, \cite{ColMin_Annals}, \cite{cheng2012eigenvalue}, \cite{cheng2013simons}, \cite{cheng2012stability}, \cite{impera_rimoldi},
and references therein for recent studies on self-shrinkers and $f$-minimal hypersurfaces.
Let us just point out a recent result relevant to our paper. Wang in \cite{MR2780753} investigated graphical self-shrinkers in $\R^n$ by studying the 
equation \eqref{mingrapheq} in the whole $\R^n$ when $f(x)= |x|^2 /4$. She proved that any smooth solution to this equation has to be a hyperplane 
improving an ealier result of Ecker and Huisken \cite{MR1025164}, where they made the extra assumption that the solution has polynomial growth. We will show that the situation is quite different when $\R^n$ is replaced by a Cartan-Hadamard manifold with strictly negative sectional curvatures and for more general $f$ satisfying some suitable assumptions. In particular, we impose that $\sup_{\bar{\Omega}\times \R}\abs{\bar{\n}f}<\infty$ which is not valid for $f(x)=\abs{x}^2/4$.

In our existence results we always assume that $f\in C^{2}(\bar{\Omega}\times\R)$ is of 
the form
\begin{equation}\label{f-form}
f(x,t)=m(x)+r(t),
\end{equation}
for discussion about this, see Section \ref{int_grad_sec}.
Our first result is the following:
\begin{thm} 
\label{existence}
Let $\Omega\subset M$ be a bounded domain with $C^{2,\alpha}$ boundary $\partial \Omega$. Suppose that $f\in C^{2}(\bar{\Omega}\times\R)$ satisfies \eqref{f-form}, with
\[
F =\sup_{\bar{\Omega}\times \R}\abs{\bar{\n}f}<\infty,\quad
\Ric_{\Omega} \geq -\dfrac{F^2}{n-1},\quad\text{and}\quad
H_{\partial \Omega} \ge F,
\]
where $Ric_{\Omega}$ stands for the Ricci curvature of $\Omega$ and 
$H_{\partial \Omega}$ for the inward mean curvature of $\partial \Omega$. Then, for all $\varphi \in C^{2,\alpha}(\partial \Omega)$,  
there exists a solution 
$u\in C^{2,\alpha}(\bar{\Omega})$ to the equation \eqref{mingrapheq} with boundary values $\varphi$.
\end{thm}
The proof of Theorem \ref{existence} is based on the Leray-Schauder method  
(see \cite[Theorem 13.8]{GilTru}), and hence requires a priori height and gradient (both 
interior and boundary) estimates for solutions.
It is worth noting already at this point that we cannot ask for the uniqueness of a 
solution if the function $f\colon M\times\R\to\R$ depends on the $t$-variable since 
comparison principles fail to hold. Indeed, an easy computation shows that for the open 
disk $B(0,2)\subset\R^2$ and $f\colon \R^2\times\R\to\R,\ f(x,t)=|(x,t)|^2/4$, both 
the upper and lower hemispheres and the disk $B(0,2)$ itself are $f$-minimal hypersurfaces 
with zero boundary values on the circle $\partial B(0,2)$. 

Thanks to the interior gradient estimate Lemma~\ref{globestim} we can weaken the regularity assumption on the boundary value function.
\begin{thm} 
\label{cont_existence}
Let $\Omega\subset M$ be a bounded domain with $C^{2,\alpha}$ boundary $\partial \Omega$. 
Suppose that $f\in C^{2}(\bar{\Omega}\times\R)$ satisfies \eqref{f-form}, with
\[
F =\sup_{\bar{\Omega}\times \R}\abs{\bar{\n}f}<\infty,\quad
\Ric_{\Omega} \geq -\dfrac{F^2}{n-1},\quad\text{and}\quad
H_{\partial \Omega} \ge F.
\]
Then, for all $\varphi \in C(\partial \Omega)$,  
there exists a solution 
$u\in C^{2,\alpha}(\Omega)\cap C(\bar{\Omega})$ to the equation \eqref{mingrapheq} with boundary values $\varphi$.
\end{thm}
Let us point out that the assumption $H_{\partial \Omega}\geq F$ is necessary. Indeed, Serrin \cite{MR0282058} has proved that the constant mean 
curvature equation
\[
\dv \frac{\nabla u}{W}=H_0
\] 
is solvable on a bounded domain $\Omega\subset\R^n$ if and only if $H_{\partial \Omega} \geq  |H_0|$;
see also \cite{MR2373228} for a related result.
If the function $f$ were a density function defined only on $M$, it might be possible to refine the assumptions
in terms of mean convexity of the boundary with respect to the weighted mean curvature and 
Bakry-Emery-Ricci tensor. 
However, in our case the function $f$ depends also on the $\R$-variable and for the a priori estimates it is
necessary to have a control of the full gradient.


Finally in Section~\ref{DPat8}, we consider the Dirichlet problem at infinity. Here we suppose that $M$ is a  
Cartan-Hadamard manifold, i.e. a complete, simply connected Riemannian manifold with non-positive sectional curvature.
We denote by $\bar{M}$ the compactification of $M$ in the cone topology
(see \cite{eberlein1973visibility}) and by $\partial_\infty M$ the asymptotic boundary of $M$. The
Dirichlet problem at infinity consists in finding solutions to \eqref{mingrapheq} in the case where $\Omega =M$ and $\partial \Omega =\partial_\infty M$. 
In order to formulate the assumptions on sectional curvatures of $M$ and on the function $f\colon M\times\R\to\R$, we first
denote by $\rho(\cdot)=d(o,\cdot)$ the (Riemannian) distance to a fixed point $o\in M$. Then we assume that sectional curvatures of $M$ satisfy
     \begin{equation}\label{curv_assump_gen}
       -(b\circ\rho)^2(x)\le K(P_x)\le -(a\circ\rho)^2(x)
     \end{equation}
   for all $x\in M$ and all $2$-dimensional subspaces $P_x\subset T_xM$, where $a$ and $b$ are smooth functions subject to conditions
   \eqref{A1}-\eqref{A7}; see Section~\ref{DPat8}. Given a smooth function 
$k\colon [0,\infty)\to [0,\infty)$, we denote by $f_k \colon [0,\infty) \to \R$
the smooth non-negative solution to the initial value problem
    \begin{equation}\label{Jacobi_eq}
      \left\{
	\begin{aligned}
	  f_k(0)&=0, \\
	  f_k'(0)&=1, \\
	  f_k''&=k^2f_k.
	\end{aligned}
      \right.
    \end{equation} 
To state the main result on the solvability of the asymptotic Dirichlet problem requires a number of definitions. First of all we assume that there exists an auxiliary smooth function $a_0\colon [0,\infty)\to (0,\infty)$ such that 
\[
 \int_1^\infty \left( \int_r^\infty \frac{ds}{f_a^{n-1}(s)} \right) a_0(r)f_a^{n-1}(r) dr <\infty.
\]    
Then we define $g\colon [0,\infty)\to [0,\infty)$ by
\begin{equation}\label{gV}
g(r)=\frac{1}{f_a^{n-1}(r)}\int_0^r a_0(t)f_a^{n-1}(t) dt.
\end{equation}
The function $g$ was introduced in \cite{mastrolia2015elliptic} where they studied some elliptic and parabolic equations with asymptotic Dirichlet 
boundary conditions on Cartan-Hadamard manifolds.
In addition to \eqref{f-form}, we assume that the function $f\in C^{2}(\bar{\Omega}\times\R)$ 
satisfies
\begin{equation}\label{f1}
\sup_{\partial B(o,r)\times\R}\abs{\nb f} \le 
\min\left\lbrace \frac{a_0(r)+(n-1)\frac{f^\prime_a (r)}{f_a(r)}g^3(r)}{\bigl(1+g^2(r)\bigr)^{3/2}}, (n-1)\frac{f^\prime_a(r)}{f_a(r)}\right\rbrace,
\end{equation}
for every $r>0$,    
and
\begin{equation}
\label{f2}
\sup_{\partial B(o,r)\times\R}\abs{\nb f}=o\left(\frac{f^{\prime}_a (r)}{f_a(r)}r^{-\ve-1}\right)
\end{equation} 
for some $\epsilon>0$ as $r\to\infty$.

The general solvability result for the asymptotic Dirichlet problem is the following. 
\begin{thm}\label{ThmMain}
   Let $M$ be a Cartan-Hadamard manifold of dimension $n\ge 2$. 
Assume that
     \begin{equation*}
       -(b\circ\rho)^2(x)\le K(P_x)\le -(a\circ\rho)^2(x)
     \end{equation*}
for all $x\in M$ and all $2$-dimensional subspaces $P_x\subset T_xM$ where $a$ and $b$ 
satisfy assumptions \eqref{A1}-\eqref{A7} and that the function $f\in C^2(M\times\R)$ on the right side of \eqref{mingrapheq} satisfies \eqref{f-form}, \eqref{f1}, and \eqref{f2}. 
   Then the asymptotic Dirichlet problem for the equation \eqref{mingrapheq} is
  solvable for any boundary data $\varphi\in C\bigl(\partial_{\infty}M\bigr)$.
\end{thm}
As a special case of the above theorem, we have:
\begin{cor}\label{thm1}
   Let $M$ be a Cartan-Hadamard manifold of dimension $n\ge 2$.  Suppose that there are constants
 $\phi>1,\ \varepsilon>0$, and $R_0>0$ such that
     \begin{equation}\label{curv_ass_minim}
       -\rho(x)^{2\left(  \phi-2\right)  -\varepsilon}\leq K(P_x)\leq-\dfrac
       {\phi(\phi-1)}{\rho(x)^{2}},
     \end{equation}
for all $2$-dimensional subspaces $P_x\subset T_{x}M$ and for all $x\in M$, with $\rho(x)\ge R_0$. Assume, furthermore, that 
 $f\in C^2(M\times\R)$ satisfies \eqref{f-form}, \eqref{f1}, and \eqref{f2}, with $f_a(t)=t$ for small $t\ge 0$ and $f_a(t)=c_1t^\phi +c_2t^{1-\phi}$ for $t\ge R_0$. 
 Then the asymptotic Dirichlet problem for equation \eqref{mingrapheq} is
   solvable for any boundary data $\varphi\in C\bigl(\partial_{\infty}M\bigr)$.
\end{cor}
In another special case we assume that sectional curvatures are bounded from above by a negative constant $-k^2$.
\begin{cor}\label{HVkor2_RT}
   Let $M$ be a Cartan-Hadamard manifold of dimension $n\ge 2$. 
   Assume that
     \begin{equation}\label{curv_assump_k}
       -\rho(x)^{-2-\varepsilon}e^{2k\rho(x)}\le K(P_x)\le -k^2
     \end{equation}
   for some constants $k>0$ and $\varepsilon>0$ and for all $2$-dimensional subspaces $P_x\subset T_x M$, with $\rho(x)\ge R_0$. Assume, furthermore, that 
 $f\in C^2(M\times\R)$ satisfies \eqref{f-form}, \eqref{f1}, and \eqref{f2}, with $f_a(t)=t$ for small $t\ge 0$ and $f_a(t)=c_1\sinh (kt) +c_2\cosh (kt)$ for $t\ge R_0$. 
   Then the asymptotic Dirichlet problem for the equation \eqref{mingrapheq} is
   solvable for any boundary data $\varphi\in C\bigl(\partial_{\infty}M\bigr)$.
\end{cor}
We refer to \cite[Ex. 2.1, Cor. 3.22]{HoVa} and to \cite[Cor. 3.23]{HoVa} for the verification of the assumptions \eqref{A1}-\eqref{A7} for the curvature bounds \eqref{curv_ass_minim} and \eqref{curv_assump_k}, respectively. We point out that, thanks to Examples \ref{a_0_exam1} and \ref{a_0_exam2}, the assumption \eqref{f1} in the above corollaries 
is weaker than \eqref{f2}
when $r\to \infty$.

Let us discuss where the assumptions \eqref{f1} and \eqref{f2} will be used in our paper. First of all, we prove Theorem~\ref{ThmMain} by extending the boundary value function $\varphi$ to $M$, exhausting $M$ by geodesic balls and solving the Dirichlet problem \eqref{mingrapheq} in each ball. In this step, the assumption 
\[
\sup_{\partial B(o,r)\times\R}\abs{\nb f} \le (n-1)\frac{f^\prime_a(r)}{f_a(r)}
\]
is used. Secondly, the other assumption in \eqref{f1},
\[
\sup_{\partial B(o,r)\times\R}\abs{\nb f} \le
\frac{a_0(r)+(n-1)\frac{f^\prime_a (r)}{f_a(r)}g^3(r)}{\bigl(1+g^2(r)\bigr)^{3/2}},
\]
is used to prove that the sequence of solutions above is uniformly bounded, thus allowing us to extract a subsequence converging towards a global solution.
Finally, we apply \eqref{f2} to prove that this global solution has proper boundary values at infinity.
Furthermore, concerning \eqref{f2}, let us mention a result of Pigola, Rigoli, and Setti in \cite{PigolaRigoliSetti}. There they considered the equation
 \[
     \dv \frac{\nabla u}{\sqrt{1+|\nabla u|^2}} =h(x),
 \]
for a function $h\in C^\infty (M)$. They proved that if $\max_M |u|<\infty$, $h$ has a constant sign,
and $M$ satisfies one of the following growth assumptions:
\begin{equation}
\label{(i)}
\vol \bigl(\partial B(o,r)\bigr)\leq C r^\alpha,\text{ for some }\alpha\geq 0
\end{equation}
or
\begin{equation}
\label{(ii)}
\vol \bigl(\partial B(o,r)\bigr)\leq C e^{\alpha r},\text{ for some }\alpha\geq 0,
\end{equation}
then necessarily we have
\[
     \liminf_{\rho(x)\to\infty} \frac{|h(x)|}{\rho^{-2}(x) \bigl(\log \rho(x)\bigr)^{-1}}=0,
\]
and
\[
     \liminf_{\rho(x)\to \infty} \frac{|h(x)|}{\rho^{-1}(x)\bigl(\log r(x)\bigr)^{-1}}=0,
\]
respectively. We notice that condition \eqref{(i)} (resp. \eqref{(ii)}) is implied by \eqref{curv_ass_minim} (resp.
\eqref{curv_assump_k}). On the other hand, assuming \eqref{curv_ass_minim} (resp. \eqref{curv_assump_k}), we notice (using Examples \ref{a_0_exam1} and \ref{a_0_exam2}) that \eqref{f2} reduces to $\sup_{\partial B(o,r)\times \R}|\bar{\nabla} f|=o(r^{-2-\varepsilon})$ (resp.  $\sup_{\partial B(o,r)\times \R}|\bar{\nabla} f|=o(r^{-1-\varepsilon})$) when $r\rightarrow \infty$. Therefore, in these cases, \eqref{f2} is almost sharp.

The paper is organised as follows: in Section~\ref{sec_estim}, we prove a priori height and gradient estimates that are needed in 
Section~\ref{sec_ex} where we apply the Leray-Schauder method and prove Theorem~\ref{existence} and \ref{cont_existence}.
Section~\ref{DPat8} is devoted to the asymptotic Dirichlet problem and proofs of Theorem~\ref{ThmMain} and Corollaries~\ref{thm1}
and \ref{HVkor2_RT}.

\section{Height and gradient estimates}\label{sec_estim}
In this section we adapt
methods from \cite{dajczer2008killing}, 
\cite{dajczer2012interior}, \cite{Korevaar}, and \cite{MR2351645} 
to obtain a priori height and gradient estimates.

\subsection{Height estimate}\label{subsec-height}
We begin by giving an a priori height estimate for solutions of the equation \eqref{mingrapheq} in a bounded
open set $\Omega\subset M$ with a $C^2$-smooth boundary assuming the estimate \eqref{Fxtra} on the function $f$. 
First we construct an upper barrier for a solution $u$ of \eqref{mingrapheq} of the form
   \begin{equation*}
     \psi(x) = \sup_{\p\Omega} \varphi + h \big( d(x) \big),
   \end{equation*}
where $d=\dist(\cdot,\p\Omega)$ is the distance from $\p\Omega$ and $h$ is a real valued function that will be determined later.
Denote by $\Omega_0$ the open set of all points $x\in \Omega$ that can be joined to $\p\Omega$ by a \define{unique} minimizing geodesic. It was shown in 
\cite{LiNirenberg} that in $\Omega_0$ the distance function $d$ has the same regularity as $\p\Omega$. 

In particular, now $d\in C^2(\Omega_0)$ and straightforward computations give
   \begin{equation*}
     \psi_i = h'  d_i \quad \text{and} \quad \psi_{i;j} = h'' d_i  d_j + h'  d_{i;j}.
   \end{equation*}
Moreover, $|\n d|^2 = d^i  d_i = 1$ and hence $d^i  d_{i;j} =0$. We also have that
   \begin{equation*}
     \sigma^{ij}  d_{i;j} = \Delta d = - H,
   \end{equation*} 
where $H=H(x)$ is the (inward) mean curvature of the level set $\{y\in\Omega_0\colon d(y)=d(x)\}$.    

Given a solution $u\in C^{2}(\Omega)$ of \eqref{mingrapheq}, 
    \begin{equation*}
      Q[u] =  \frac{1}{W} \left( \sigma^{ij} - \frac{u^i u^j}{W^2} \right)  u_{i;j} - \ang{\nb f,\nu} = 0,
    \end{equation*}
we define $b\colon\Omega\to\R$ by
\begin{equation}\label{b-def}
b(x)=\left\langle \bar{\nabla}f\bigl(x,u(x)\bigr),\nu(x)\right\rangle,
\end{equation}
where $\nu(x)$ is the downward pointing unit normal to the graph of $u$ at $\big(x,u(x)\big)$.  
Next we define an operator
   \begin{equation*}
   \tilde{Q}[v] =  \frac{1}{W} \left( \sigma^{ij} - \frac{v^i v^j}{W^2} \right)  v_{i;j} - b,
   \end{equation*}
where $W = \sqrt{1+|\nabla v|^2}$ and $b$ does not depend on $v$. The reason to define such an operator is that
it allows us to use the comparison principle whereas the operator $Q$ need not satisfy the required assumptions,
see e.g. \cite[Theorem 10.1]{GilTru}.
Then for a point $x\in \Omega_0$ we obtain
   \begin{align} \label{operQ}
    \tilde{Q}[\psi] + b &= \frac{1}{W} \left( \sigma^{ij} - \frac{(h')^2 d^i d^j}{W^2} \right)
       \left( h'' d_i  d_j + h'  d_{i;j} \right) \nonumber \\
    &= \frac{1}{W} \left( h'' + h' \Delta d - \frac{(h')^2 h''}{W^2} \right) \nonumber \\
    &= \frac{1}{W} \left( \frac{h''}{W^2} - h' H(x) \right) \nonumber \\
    &= \frac{h''}{W^3} - \frac{h'}{W} H(x),
   \end{align}
where we used that $W^2 = 1+(h')^2$.

Next we impose an extra condition on the function $f\colon M\times\R\to\R$ by assuming that 
\begin{equation}\label{Fxtra}
\sup_{s\in\R}\abs{\bar{\nabla}f(x,s)}\le H(x)
\end{equation}
for all $x\in\Omega_0$. Hence $\abs{b(x)}\le H(x)$ for all $x\in\Omega_0$. By choosing
   \[
     h= \frac{e^{AC}}{C}\big(1-e^{-Cd}\big),
   \]
where $A = \diam(\Omega)$ and 
\[
C>\sup_{\Omega_0\times\R}\abs{\bar{\nabla}f}
\]
 is a constant, we obtain
   \[
     h' = e^{C(A-d)} \ge 1 \quad \text{and} \quad h'' = -Ch',
   \]
and so
 \begin{align*}
  \tilde{Q}[\psi] &= -\frac{Ch'}{W^3} - \frac{h'H}{W} - b \\
  &< - \abs{b}\left(\frac{h'}{W^3} + \frac{h'}{W}-1 \right) \\
  &\le 0.
 \end{align*}
Therefore we have  
   \begin{equation*} \begin{cases}
     \tilde{Q}[\psi] <  0 = \tilde{Q}[u] = Q[u] \quad \text{in } \Omega_0 \\
     \psi|\p\Omega \ge u|\p\Omega = \varphi|\p\Omega. \end{cases}
   \end{equation*}
Next we observe that $\psi \ge u$ in $\bar{\Omega}$. Assume on the contrary that the continuous function $u-\psi$ attains its positive maximum 
at an interior point $x_0 \in \Omega$. As in \cite[p. 795]{MR2351645} (see also \cite[pp. 239-240]{dajczer2008killing}),
we conclude that, in fact, $x_0$ is an interior point of $\Omega_0$ that leads to a contradiction
with the comparison principle \cite[Theorem 10.1]{GilTru} which states that $u-\psi$ can not attain its maximum in 
the open set $\Omega_0$.

Similarly we deduce that $\psi^{-}$,
\[
\psi^{-}(x)=\inf_{\p\Omega} \varphi - h \big( d(x) \big),
\]
is a lower barrier for $u$, i.e. $\psi^{-}\le u$ in $\bar{\Omega}$.
These barriers imply the following height estimate for $u$.
\begin{lem} \label{heightestim}
Let $\Omega\subset M$ be a bounded open set with a $C^2$-smooth boundary and suppose that
\begin{equation}\label{Fxtra-again}
 \sup_{s\in\R}\abs{\bar{\nabla}f(x,s)}\le H(x)
\end{equation}
in $\Omega_0$. Let $u \in C^2(\Omega)\cap C(\bar{\Omega})$ be a solution of $Q[u] = 0$ with $u|\p\Omega = \varphi$.
Then there exists a constant $C=C(\Omega)$
 such that
   \[
     \sup_{\Omega} |u| \le C + \sup_{\p\Omega} |\varphi|.
   \]
\end{lem}

\subsection{Boundary gradient estimate}

In this subsection we will obtain an a priori boundary gradient estimate for the Dirichlet problem \eqref{mingrapheq}. We assume that $\Omega\subset M$ is a bounded open set with a $C^2$-smooth boundary and that $\Omega_\ve$ 
is a sufficiently small tubular neighborhood of $\p\Omega$ so that the distance function 
$d$ from $\p\Omega$ is $C^2$ in $\Omega_\ve\cap \bar{\Omega}$. Furthermore, we assume that the (inward) mean curvature $H=H(x)$ of the level set  $\{y\in\bar{\Omega}_0\colon d(y)=d(x)\}$
satisfies
\begin{equation}\label{H-f}
H(x)\ge\sup_{s\in\R}\abs{\bar{\nabla}f(x,s)}:=F(x)
\end{equation}
for all $x\in\Omega_\ve\cap\bar{\Omega}$.
Next we extend the boundary function $\varphi$, which is assumed to be $C^2$-smooth, to $\Omega_\ve$ by setting $\varphi\big(\exp_{y}t \n d(y) \big) = \varphi(y)$, for $y\in\p\Omega$, where $\n d(y)$ is the unit \define{inward} normal to $\p\Omega$ at $y\in\p\Omega$. We will construct barriers of the form $w+\varphi$, 
where $w= \psi\circ d$ and $\psi$ is a real function that will be determined later. 

We denote
   \begin{equation}\label{aijdef}
     a^{ij}=a^{ij}(x,\n v) = \frac{1}{W} \left( \sigma^{ij} - \frac{v^iv^j}{W^2} \right),\quad
       W=\sqrt{1+\abs{\n v}^2},
   \end{equation}
and, given a solution $u\in C^2(\Omega)\cap C^1(\bar{\Omega})$ of \eqref{mingrapheq}, we define an operator
\[
\tilde{Q}[v]=a^{ij}(x,\n v)v_{i;j}-b,
\]
with $b$ as in \eqref{b-def}.

The matrix $a^{ij}(x,\n v)$ is positive definite with eigenvalues 
   \begin{equation}\label{eigval}
     \lambda = \frac{1}{W^3} \quad \text{and} \quad \Lambda = \frac{1}{W}
   \end{equation}
with multiplicities $1$ and $n-1$ corresponding respectively to the directions parallel and orthogonal
to $\n v$. Hence a simple estimate gives 
\begin{equation} \label{Qestim}
     \tilde{Q}[w+\varphi] = a^{ij}( w_{i;j} +  \varphi_{i;j}) - b 
\le a^{ij}  w_{i;j} + \Lambda \norm{\varphi}_{C^2} - b,
\end{equation}   
 where $a^{ij}=a^{ij}(x, \n w + \n \varphi),\ 
 \Lambda=(1+\abs{\n w +\n \varphi}^2)^{-1/2}$,
and $\norm{\varphi}_{C^2}$ denotes the $C^{2}(\Omega_{\ve})$-norm of $\varphi$.
Since in $\Omega_\ve\cap\bar{\Omega}$ we have $|\n d|^2 = d^i d_i = 1$, $d^i d_{i;j} =0$, and $\ang{\n d,\n \varphi}=0$, straightforward computations give that
   \begin{align*}
     \Delta w = \psi'' &+ \psi' \Delta d, \\
     w^i w^j  w_{i;j} &= (\psi')^2 \psi'', \\
     w^i \varphi^j  w_{i;j} &= \psi' \psi'' \ang{\n d, \n \varphi} = 0,
   \end{align*}
and also
   \[
     \varphi^i \varphi^j  w_{i;j} =\psi''\ang{\n \varphi,\n d}^{2}  +\psi' \varphi^i\varphi^j  d_{i;j} =\psi' \varphi^i\varphi^j  d_{i;j}.
   \]
With these, and noticing that now $W^2= 1+ (\psi')^2 + |\n \varphi|^2$, we obtain
   \begin{align} \label{aijweq}
    a^{ij}  w_{i;j} &= \frac{\psi' \Delta d}{W} + \frac{\psi''(1+ |\n \varphi|^2)}{W^3} -
    \frac{\psi' \varphi^i \varphi^j  d_{i;j}}{W^3}.
   \end{align}
Putting \eqref{Qestim} and \eqref{aijweq} together, we arrive at
\begin{equation}\label{1stestim}
    \tilde{Q}[w+\varphi] \le \frac{\psi' \Delta d}{W} + \frac{\psi''(1+ |\n \varphi|^2)}{W^3} -
    \frac{\psi' \varphi^i \varphi^j  d_{i;j}}{W^3} + \Lambda \norm{\varphi}_{C^2} +F.
\end{equation}

Next we define
   \[
     \psi(t) = \frac{C\log(1+Kt)}{\log(1+K)},
   \]
where the constants 
\[
C \ge 2\big(\max_{\bar{\Omega}}|u| +\max_{\bar{\Omega}}|\varphi|\big),
\]
$K\ge (1-2\ve)\ve^{-2}$, and $\ve\in (0,1/2)$ will be chosen later. Then
\[
\psi(\ve)=\frac{C\log(1+K\ve)}{\log(1+K)}\ge C/2
\]
and we have
 \begin{equation}\label{inbdry}
(w+\varphi)|\Gamma_\ve =  \psi(\ve) + \varphi|\Gamma_\ve 
 \ge u|\Gamma_\ve
\end{equation}
on the ``inner boundary'' $\Gamma_{\ve}=\{x\in\Omega\colon d(x)=\ve\}$ of $\Omega_\ve$.
On the other hand,
 \begin{equation}\label{outbdry}
     (w+\varphi)|\p\Omega =  u|\p\Omega.
 \end{equation}
 
We claim that $\tilde{Q}[w+\varphi]\le 0$ in $\Omega_\ve\cap\Omega$ if $C,\ K$, and $\ve$ 
are properly chosen. All the computations below will be done in $\Omega_\ve\cap\Omega$ without further notice.
We first observe that
 \[
     \psi'(t) = \frac{CK}{(1+Kt)\log(1+K)} \quad \text{and} \quad 
     \psi''(t) = -\frac{\log(1+K)\psi'(d)^2}{C},
   \]
and therefore we have
\begin{align}\label{2ndestim}
W  \tilde{Q}[w+\varphi] &\le \big(W-\psi'\big)H 
- \frac{\log(1+K)}{C}\left(\frac{\psi'}{W}\right)^2\big(1+|\nabla\varphi|^2\big)\nonumber\\
&\quad + \norm{\varphi}_{C^2}
+ |\nabla\varphi|^2 H
\end{align}
by \eqref{H-f}, \eqref{eigval}, and \eqref{1stestim}. We estimate
\[
\psi'\ge\frac{CK}{(1+K\ve)\log(1+K)}=\frac{C}{(\ve+1/K)\log(1+K)} = 1
\]
and consequently, 
\begin{align*}
\frac{\psi'}{W} &\ge c_1=c_1\big(\max_{\bar{\Omega}}|\nabla\varphi|\big)>0\\
\noalign{and}
 W-\psi'&\le c_2=c_2\big(\max_{\bar{\Omega}}|\nabla\varphi|\big)
\end{align*}
by choosing $C= (\ve+1/K)\log(1+K)$.
The claim $\tilde{Q}[w+\varphi]\le 0$ now follows from \eqref{2ndestim} since
\[
\frac{\log(1+K)}{C}=\frac{1}{\ve +1/K}\ge \frac{c_2 H+\norm{\varphi}_{C^2}+|\nabla\varphi|^2 H}{c_1^2\big(1+|\nabla\varphi|^2\big)}
\]
by choosing sufficiently small $\ve$ and large $K$ depending only on $\max_{\bar{\Omega}}|u|,\ \norm{\varphi}_{C^2}$, and $H_{\p\Omega}$.

Hence
\[
\tilde{Q}[w+\varphi]\le 0=\tilde{Q}[u],
\]
and therefore 
$w+\varphi$ is an upper barrier in $\Omega_\ve\cap\Omega$. Similarly, $-w+\varphi$ is a lower barrier.  
Together these barriers imply that
\[
\abs{\n u}\le \abs{\n w} +\abs{\n \varphi}= \psi'(0)+\abs{\n \varphi}=\frac{CK}{\log(1+K)}+\abs{\n \varphi}
\]
on $\p\Omega$.

We have proven the following boundary gradient estimate.
   \begin{lem}\label{boundestim}
   Let $\Omega\subset M$ be a bounded open set with a $C^2$-smooth boundary and suppose that
\begin{equation}\label{H-f-again}
 \sup_{s\in\R}\abs{\bar{\nabla}f(x,s)}\le H(x)
\end{equation}
in some tubular neighborhood of $\p\Omega$. Let $u\in C^2(\Omega) \cap C^1(\bar{\Omega})$ be a solution to $Q[u] =0$ with $u|\p\Omega =\varphi\in C^{2}(\p\Omega)$. Then
       $$
         \max_{\p\Omega} |\n u| \le C,
       $$
     where $C$ is a constant depending only on $\sup_{\bar{\Omega}} |u|,\ H_{\p\Omega}$, 
     and $\norm{\varphi}_{C^2(\p\Omega)}$.
   \end{lem}

\subsection{Interior gradient estimate}\label{int_grad_sec}

In this subsection we will assume that $u$ is a $C^3$ function. The
elliptic regularity theory will guarantee that the estimate holds also for $C^{2,\alpha}$ solutions. We also assume that $f\colon M\times\R\to\R$ is of the form
\[
f(x,t)=m(x)+r(t).
\]
In particular, all ``space'' derivatives
\[
f_i=\frac{\p f}{\p x_i},\quad i=1,\ldots,\dim M,
\]
are independent of $t$; $f_{it}=f_{ti}=0$.
Since we are dealing with the Riemannian product $M\times\R$ carrying the parallel vector field $\p_t$, this
assumption is not so unnatural. Also, due to the fact that $f$ depends also on the $\R$-variable, it seems to be 
hard, if not impossible, to adapt other known approaches  (see e.g. \cite{Dajczer2016}) to get rid of this assumption.


For an open set $\Omega\subset M$, we denote $i(\Omega)=\inf_{x\in\Omega}i(x)$, where
$i(x)$ is the injectivity radius at $x$. Thus $i(\Omega)>0$ if $\Omega\Subset M$ is 
relatively compact.  Furthermore, we denote by $R_{\Omega}$ the Riemannian curvature 
tensor in $\Omega$.

\begin{lem}\label{globestim}
Let $u\in C^3 (\Omega)$ be a solution of \eqref{mingrapheq} with $u<m_u$ for some constant $m_u<\infty$.
\begin{enumerate}
\item[(a)] For every ball $B(o,r)\subset\Omega$, there exists a constant 
\[
L=L\big(u(o),m_u,r,R_{\Omega},\norm{f}_{C^2(\Omega\times (-\infty,m_u))}\big)
\] 
such that $\abs{\nabla u(o)}\le L$.
\item[(b)] If, furthermore, $u\in C^1(\bar{\Omega})$, we have a global gradient bound 
\[
\abs{\nabla u(o)}\le L
\]
for every $o\in\bar{\Omega}$, with
\[
L=L\big(u(o),m_u,i(\Omega),\diam(\Omega),R_{\Omega},\norm{f}_{C^2(\Omega\times (-\infty,m_u))},\max_{\partial\Omega}\abs{\nabla u})\big)<\infty.
\]
\end{enumerate}
\end{lem}
\begin{proof}
We apply the method due to Korevaar and Simon \cite{Korevaar}; see also \cite{dajczer2012interior}. 
Let $0<r\le\min\{i(\Omega),\diam(\Omega)\},\ o\in\Omega$, and let $\eta$ be a continuous non-negative function on $M$, vanishing outside $B(o,r)$ and smooth 
whenever positive. The function $\eta$ will be specified later. Define 
  \[
    h=\eta W
  \] 
and assume first that $h$ attains its maximum at an interior point $p\in B(o,r)\cap\Omega$.
The case $p\in B(o,r)\cap\partial\Omega$ and $u\in C^1(\bar{\Omega})$ will be commented at the end of the proof.

We will first prove an upper bound for $\abs{\nabla u(p)}$. Therefore we may assume that $\abs{\nabla u(p)}\ne 0$.
We choose normal coordinates at $p$ so that $\partial_1=\nabla u/\abs{\nabla u}$ at $p$.
All the computations below will be made at $p$ without further notice. Thus we have $\sigma_{ij}=\sigma^{ij}=\delta^{ij},\ u_1=u^1=\abs{\nabla u}$,
and $u_{j}=u^{j}=0$ for $j>1$. Furthermore,
  \[
      a^{ij}=\frac{1}{W}\left(\delta^{ij}-\frac{\abs{\nabla u}^2\delta^{1i}\delta^{1j}}{W^2}\right),
  \]
and therefore $a^{11}=W^{-3},\ a^{ii}=W^{-1}$ for $i>1$, and $a^{ij}=0$ if $i\ne j$. At the maximum point $p$, we have
$h_i=0$ and $h_{i;i}\le 0$ for all $i$. Hence 
  \begin{equation}\label{etaWsym}
      \eta_i W=-\eta W_i
  \end{equation}
and
  \[
      a^{ij}h_{i;j}=a^{ii}h_{i;i}=a^{ii}\left(W\eta_{i;i}+2\eta_{i}W_{i}+\eta W_{i;i}\right)\le 0.
  \]
With \eqref{etaWsym} we can write this as
  \begin{equation}\label{Waeta}
      Wa^{ii}\eta_{i;i}+\frac{\eta a^{ii}}{W}\left(WW_{i;i}-2(W_{i})^2\right)\le 0.
  \end{equation}
We have
  \[
    W_i=\frac{u^k u_{k;i}}{W}=\frac{\abs{\nabla u}u_{1;i}}{W}
  \]
and from \eqref{unitnormal} we see that the $ k^{\rm th} $ component of the unit normal is 
  \[
    \nu^k =\frac{u^k}{W}=\frac{\abs{\nabla u}\delta^{k1}}{W}.
  \]
To scrutinize the second order differential inequality \eqref{Waeta}, we first compute
  \begin{align*}
    a^{ii}W_{i;i}&=a^{ii}\bigl(W^{-1}u^k u_{k;i}\bigr)_{;i}\\
    &=-\frac{a^{ii}\abs{\nabla u}u_{1;i}W_i}{W^2}+
    \frac{a^{ii}u^{k}_{;i}u_{k;i}}{W}
    +\frac{a^{ii}\abs{\nabla u}u_{1;ii}}{W}\\
    &=-\frac{a^{ii}\abs{\nabla u}^2(u_{1;i})^2}{W^3}+
    \frac{a^{ii}u^{k}_{;i}u_{k;i}}{W}
    +\frac{a^{ii}\abs{\nabla u}u_{1;ii}}{W}\\
    &=\frac{a^{ii}(u_{1;i})^2}{W^3}+\frac{a^{ii}\sum_{k\ne 1}(u_{k;i})^2}{W}+
    \frac{a^{ii}\abs{\nabla u}u_{1;ii}}{W}.
  \end{align*}
Hence
  \begin{equation}\label{W2deriv}
      Wa^{ii}W_{i;i}=A+a^{ii}\abs{\nabla u}u_{1;ii},
  \end{equation}
where
  \[
      A=a^{ii}(u_{1;i})^2W^{-2}+a^{ii}\sum_{k\ne 1}(u_{k;i})^2\ge 0.
  \]
Using the Ricci identities for the Hessian of $u$ we get
  \[
      u_{k;ij}=u_{i;kj}=u_{i;jk}+R^{\ell}_{kji}u_{\ell},
  \] 
where $R$ is the curvature tensor in $M$. This yields
  \begin{equation}\label{ricciaplic}
      \abs{\nabla u}a^{ii}u_{1;ii}=\abs{\nabla u}a^{ii}u_{i;i1}+\abs{\nabla u}^2 a^{ii}R^{1}_{1ii}.
  \end{equation}
To compute $\abs{\nabla u}a^{ii}u_{i;i1}$, we first observe that 
  \[
      Wa^{ij}u_{i;j}=
      Wa^{ii}u_{i;i}=\langle\bar{\nabla}f,(\nabla u,-1)\rangle=f_{i}u^{i}-f_t.
   \]
Since
  \begin{align*}
    \nu^1\bigl(Wa^{ij}\bigr)_{;1}u_{i;j}&=\nu^1\bigl(\sigma^{ij}-u^iu^j W^{-2}\bigr)_{;1}u_{i;j}\\
    &=-\frac{\abs{\nabla u}}{W}\left(\frac{2u^i u^{j}_{;1}}{W^2}-\frac{2u^i u^j W_1}{W^3}\right)u_{i;j}\\
    &=-\frac{2\abs{\nabla u}^2}{W^3}u^{j}_{;1}u_{1;j}+\frac{2\abs{\nabla u}^4(u_{1;1})^2}{W^5}\\
    &=-\frac{2\abs{\nabla u}^2}{W^3}\left(\sum_{i}(u_{1;i})^2-\frac{\abs{\nabla u}^2}{W^2}(u_{1;1})^2\right)\\
    &=-\frac{2\abs{\nabla u}^2a^{ii}(u_{1;i})^2}{W^2}\\
    &=-2a^{ii}(W_i)^2,
  \end{align*}
we obtain
  \begin{align}\label{WaW2deriv}
    \abs{\nabla u}a^{ii}u_{i;i1} &=\abs{\nabla u}a^{ij}u_{i;j1}
    =\nu^{1}Wa^{ij}u_{i;j1}\nonumber\\
    &=\nu^1\bigl(Wa^{ij}u_{i;j}\bigr)_{;1}-\nu^1\bigl(Wa^{ij}\bigr)_{;1}u_{i;j}\nonumber\\
    &=\nu^1\bigl(f_{i}u^{i}-f_t\bigr)_{;1}+2a^{ii}(W_i)^2\nonumber\\
    &=\nu^1\bigl(f_i u^{i}_{;1}+(f_\ell)_{;1}u^{\ell}-(f_t)_{;1}\bigr)+2a^{ii}(W_i)^2\\
    &=\frac{\abs{\nabla u}}{W}\bigl(f_i u^{i}_{;1}+(f_1)_{;1}u^1-f_{tt}u^1\bigr)+2a^{ii}(W_i)^2\nonumber\\
    &=W_{i}f^{i}+\frac{f_{11}\abs{\nabla u}^2}{W}-\frac{f_{tt}\abs{\nabla u}^2}{W}+2a^{ii}(W_i)^2,\nonumber
  \end{align}
where we have denoted $(f_j)_{;1}=\bigl(x\mapsto f_j(x,u(x))\bigr)_{;1}$ and used the assumption $f_{it}=f_{ti}=0$.
Putting together \eqref{etaWsym}, \eqref{W2deriv}, \eqref{ricciaplic}, and \eqref{WaW2deriv} we can estimate the inequality \eqref{Waeta} as
  \begin{align}\label{WaetaEstim}
    0&\ge Wa^{ii}\eta_{i;i}+\frac{\eta a^{ii}}{W}\left(WW_{i;i}-2(W_{i})^2\right)\nonumber\\
&=Wa^{ii}\eta_{i;i} + \frac{\eta}{W}\left( A+\abs{\nabla u}a^{ii}u_{1;ii}-\abs{\nabla u}a^{ii}u_{i;i1} +W_if^i 
+\frac{\abs{\nabla u}^2 \left(f_{11}-f_{tt}\right)}{W}\right)\nonumber\\
    &=Wa^{ii}\eta_{i;i}+\eta\left(\frac{A}{W}+\frac{\abs{\nabla u}^2 a^{ii}R^{1}_{1ii}}{W}
    +\frac{\abs{\nabla u}^2(f_{11}-f_{tt})}{W^2}\right)-f^{i}\eta_{i}\\
    &\ge Wa^{ii}\eta_{i;i}-f^{i}\eta_{i} -N\eta,\nonumber
\end{align}
where $N$ is a positive constant depending only on the curvature tensor in $\Omega$ and
the $C^2$-norm of $f$ in the cylinder $\Omega\times (-\infty,m_u)$. Note that 
$A\ge 0$, $a^{11}=W^{-3}$, and $a^{ii}=W^{-1}$ for $i\ne 1$.

Now we are ready to choose the function $\eta$ as 
  \[
      \eta(x)=g\bigl(\phi(x)\bigr), 
  \]
where
  \[
      g(t)=e^{C_1 t}-1
  \]
with a positive constant $C_1$ to be specified later and
  \[
      \phi(x)=\left(1-r^{-2}d^2 (x)+ C\bigl(u(x)-m_u\bigr)\right)^{+}.
  \]
Here $d(x)=d(x,o)$ is the geodesic distance to $o$ and 
  \[
      C=\frac{-1}{2\bigl(u(o)-m_u\bigr)}>0.
  \]
It follows that $\eta$ fulfils the requirements and, moreover, $\eta(o)=e^{C_1/2}-1>0$.
We have 
  \begin{equation}\label{etai}
      \eta_i=\left(-r^{-2}(d^2)_i +Cu_i\right)g'
  \end{equation}
and
  \begin{equation}\label{etaij}
      \eta_{i;j}=\left(-r^{-2}(d^2)_{i;j} +Cu_{i;j}\right)g'
      +\left(-r^{-2}(d^2)_i +Cu_i\right)\left(-r^{-2}(d^2)_j +Cu_j\right)g''.
  \end{equation}
A straightforward computation gives the estimate
  \begin{align}\label{Waprodestim}
      Wa^{ii}&\bigl(r^{-2}(d^2)_i -Cu_i\bigr)^2 
      = Wa^{ii}\bigl(r^{-4}(d^2)^2_i -2Cr^{-2}(d^2)_i u_i +C^2(u_i)^2\bigr)  \nonumber\\
      &= r^{-4}\abs{\nabla d^2}^2 - 2Cr^{-2}\langle\nabla d^2,\nabla u\rangle 
      +C^2\abs{\nabla u}^2  \nonumber\\
      &\qquad -\frac{\langle\nabla d^2,\nabla u\rangle^2}{r^4 W^2}
      +\frac{2C\abs{\nabla u}^2\langle\nabla d^2,\nabla u\rangle}{r^2 W^2}
      -\frac{C^2\abs{\nabla u}^4}{W^2}  \nonumber\\
      &=\frac{C^2\abs{\nabla u}^2}{W^2}-\frac{2C\langle\nabla d^2,\nabla u\rangle}{r^2 W^2}
      +\frac{1}{r^4}\left(\abs{\nabla d^2}^2-\frac{\langle\nabla d^2,\nabla u\rangle^2}{W^2}\right)  \nonumber\\
      &\ge \frac{C^2\abs{\nabla u}^2}{W^2}-\frac{2C\langle\nabla d^2,\nabla u\rangle}{r^2 W^2}.
  \end{align}
Next we observe that
  \begin{align}\label{waiiest}
      Wa^{ii}\bigl(-r^{-2}(d^2)_{i;i}+Cu_{i;i}\bigr)&=-r^{-2}Wa^{ii}(d^2)_{i;i}+CWa^{ii}u_{i;i}\\
      &=-r^{-2}\Delta d^2+\frac{\abs{\nabla u}^2}{r^2 W^2 }(d^2)_{1;1}+CW\langle\bar{\nabla}f,\nu\rangle\nonumber\\
      &=-r^{-2}\Delta d^2+\frac{\abs{\nabla u}^2}{r^2 W^2 }\Hess d^2(\partial_1,\partial_1)
	+CW\langle\bar{\nabla}f,\nu\rangle.\nonumber
  \end{align}
Putting together \eqref{etai}, \eqref{etaij}, \eqref{Waprodestim}, and \eqref{waiiest} we 
obtain
  \begin{align*}
      Wa^{ii}\eta_{i;i}&\ge g'\left(-r^{-2}\Delta d^2 +\frac{\abs{\nabla u}^2}{r^2W^2}\Hess d^2 (\partial_1,\partial_1)
      +CW\langle\bar{\nabla}f,\nu\rangle\right)\\
      &\qquad +g''\left(\frac{C^2\abs{\nabla u}^2}{W^2}-\frac{2C}{r^2W^2}\langle\nabla u,\nabla d^2\rangle\right).
  \end{align*}
Hence, by \eqref{WaetaEstim}, we have
  \begin{equation}\label{g2nd}
      g''\left(\frac{C^2\abs{\nabla u}^2}{W^2}-\frac{2C}{r^2W^2}\langle\nabla u,\nabla d^2\rangle\right)
      + g'P-Ng\le 0,
  \end{equation}
where 
  \[
      P=\frac{\abs{\nabla u}^2}{r^2W^2}\Hess d^2 (\partial_1,\partial_1)- \frac{\Delta d^2}{r^2}+
      \frac{f^i(d^2)_i}{r^2} - Cf_t.
  \]
It is easy to see that 
  \[
      \abs{P}\le
      \frac{\abs{\Hess d^2(\partial_1,\partial_1)}+\abs{\Delta d^2}}{r^2}+\frac{2d\abs{f^i d_i}}{r^2}+C\abs{f_t}\le C_0,
  \]
with a constant $C_0=C_0(u(o)-m_u,r,R_{\Omega},\norm{f}_{C^1})$.

In order to obtain an upper bound for $\abs{\nabla u(p)}$, we suppose that 
  \[
      \abs{\nabla u(p)}\ge\frac{16\bigl(m_u-u(o)\bigr)}{r}
  \]
and derive a contradiction.
Since $\abs{\nabla d^2(p)}\le 2r$, we see that
  \[
      \abs{\nabla u(p)}\ge\frac{4\abs{\nabla d^2(p)}}{Cr^2}
  \]
and hence we have 
  \[
      \abs{\nabla u}^2-\frac{2}{Cr^2}\langle\nabla u,\nabla d^2\rangle\ge\frac{1}{2}\abs{\nabla u}^2
  \]
at $p$. Therefore there exists a constant $D$ depending only on $m_u-u(o)$ and $r$ such that
  \[
      \frac{C^2}{W^2}\left(\abs{\nabla u}^2-\frac{2}{Cr^2}\langle\nabla u,\nabla d^2\rangle\right)\ge D>0.
  \]
But now, taking $C_1=C_1(C_0,D,N)$ large enough, we obtain 
  \[
      Dg''(\phi(p))-C_0 g'(\phi(p))-Ng(\phi(p))=(DC_1^2-C_1C_0-N)e^{C_1\phi(p)}+N>0
  \]
which is a contradiction with \eqref{g2nd}. Hence we have
  \[
      \abs{\nabla u(p)}<\frac{16\bigl(m_u-u(o)\bigr)}{r}
  \]
which implies  
  \[
      W(p)\le C_2=1+\frac{16\bigl(m_u-u(o)\bigr)}{r}.
  \]
Since $p$ is a maximum point of $h=\eta W$, we have
  \[
      \left(e^{C_1/2}-1\right)W(o)=\eta(o)W(o)\le\eta(p)W(p)\le C_2\left(e^{C_1}-1\right).
  \]
This proves the case (a).

For the case (b), we assume, in addition, that $u\in C^1(\bar{\Omega})$ and we fix $r=\min\{i(\Omega),\diam(\Omega)\}>0$.
Let $o\in\bar{\Omega}$ and $h=\eta W$ be as above with the same constant $C_1$.  
If a maximum point $p$ of $h$ is an interior point of $\Omega$, the proof for the case (a) applies and we have a desired upper bound
for $\abs{\nabla u(o)}$. On the other hand, if $p\in\partial\Omega$ we have an upper bound 
  \[
      \abs{\nabla u(p)}\le \max_{\partial\Omega}\abs{\nabla u}
  \]
and again we are done. 
\end{proof}

\section{Existence of $f$-minimal graphs}\label{sec_ex}
In this section we will prove Theorem~\ref{existence} and \ref{cont_existence}. Throughout this section we assume that
$\Omega\subset M$ is a bounded open set with $C^{2,\alpha}$ boundary $\p\Omega$. As in 
Subsection~\ref{subsec-height} we denote by $\Omega_0$ the open 
set of all those points of $\Omega$ that can be joined to $\p\Omega$ by a unique minimizing geodesic. We start with the following lemma from
\cite[Lemma 4.2]{MR2351645}; see also \cite[Lemma 5]{dajczer2008killing}. Since our definition of the mean curvature differs by a multiple constant from 
the one used in \cite{MR2351645} and \cite{dajczer2008killing}, we sketch the proof. 
\begin{lem} \label{meancurvlemma}
Let $F=\sup\{\abs{\bar{\n}f(x,s)}\colon (x,s)\in\bar{\Omega}\times\R\}<\infty$ and suppose that 
$\Ric_{\Omega} \geq -F^2/(n-1)$ and $H_{\partial \Omega} \ge F$.
Then for all $x_0\in \Omega_0$ the inward mean curvature $H(x_0)$ of the level set $\{x\in\Omega \colon d(x) = d(x_0)\}$ 
passing through $x_0$ has a lower bound $H(x_0)\geq F$.
\end{lem} 
\begin{proof}
Denote by $H(t)$ the inward mean curvature of the level set $\Gamma_t=\{x\in\Omega \colon d(x) = t\}$ at the point
which lies on the unit speed minimizing geodesic $\gamma$ joining $\gamma(0) \in \p\Omega$ to $x_0$. Denote by $N=\dot{\gamma}_{t}$ the
inward unit normal to $\Gamma_t$ and by $S_t$ the shape operator, $S_t(X) = -\nabla_{X} N$, of the level set $\Gamma_t$. As in \cite{dajczer2008killing} 
we obtain the Riccati equation
   \[
     S^{\prime}_{t} = S^{2}_{t} + R_t,
   \]
where $R_t=R(\cdot,\dot{\gamma}_t)\dot{\gamma}_t$.
Trace and derivative commute, but because of the term $S^{2}_{t}$, we need to substitute $s = \tr S_t/(n-1)$
in order to get similar differential equation for the traces. Hence we have
   \[
     s' = s^2 + r,
   \]
where $r$ satisfies $r \ge \Ric(\dot{\gamma}_t,\dot{\gamma}_t)/(n-1)$. In other words,
   \[
     \frac{\tr S^{\prime}_{t}}{n-1} \ge \left(\frac{\tr S_t}{n-1} \right)^2 + \frac{1}{n-1} \Ric(\dot{\gamma}_t,\dot{\gamma}_t).
   \]

Since $H(t) = \tr S_t$, we obtain the estimate
   \[
     \frac{H'(t)}{n-1} \ge \left(\frac{H(t)}{n-1}\right)^2 +
         \frac{1}{n-1} \Ric\big(\dot{\gamma}_t,\dot{\gamma}_t\big) \ge
     \frac{H^2(t)}{(n-1)^2} - \frac{F^2}{(n-1)^2}.
   \]
On the boundary we have $H(0) = H_{\p\Omega} \ge F$ 
which implies that $H'(t) \ge 0$ and hence the claim follows.
\end{proof}  

\begin{proof}[Proof of Theorem \ref{existence}]
In order to prove Theorem~\ref{existence} we assume that the given boundary value function is extended to a function 
$\varphi\in C^{2,\alpha}(\bar{\Omega})$ and we consider a family of Dirichlet problems 
\begin{equation} \label{eqfamily}
 \begin{cases}
\dv \dfrac{\nabla u}{\sqrt{1+|\nabla u|^2}} - \tau\ang{\nb f,\nu}=0 \quad \text{in } \Omega, \\
  u = \tau\varphi \quad \text{in } \p\Omega,\ 0\le \tau\le 1.
 \end{cases}
\end{equation}
By Lemma~\ref{meancurvlemma},
\[
H(x)\ge F\ge \sup_{\bar{\Omega}\times\R}\abs{\bar{\n}(\tau f)}
\]
for all $x\in\Omega_0$ and for all $\tau\in [0,1]$. 
Hence \emph{if} $u\in C^{2,\alpha}(\bar{\Omega})$ is a solution of \eqref{eqfamily} for some $\tau\in [0,1]$, it follows from
Lemmata~\ref{heightestim}, \ref{boundestim}, and \ref{globestim} that
 \[
\norm{u}_{C^1(\bar{\Omega})}\le C 
 \] 
 with a constant $C$ that is independent of $\tau$. The Leray-Schauder method \cite[Theorem 13.8]{GilTru} then yields a solution to the Dirichlet
 problem \eqref{eqfamily} for all $\tau\in [0,1]$. In particular, with $\tau=1$ we obtain a solution to the original Dirichlet problem.
 \end{proof}

\begin{proof}[Proof of Theorem \ref{cont_existence}]
Let $\varphi \in C(\p\Omega)$ and let $\varphi^{\pm}_k\in C^{2,\alpha}(\p\Omega)$ be two monotonic
sequence converging uniformly on $\p\Omega$ to $\varphi$ from above and from below, respectively. Denote
        $$
                F^+ = \sup_{\bar{\Omega}\times\R} |\nb f| \quad \text{and} \quad F^- = -F^+.
        $$
By Theorem~\ref{existence} there are functions $u^{\pm}_{k},v^{\pm}_{k}\in C^{2,\alpha}(\bar{\Omega})$ such that
$u^{\pm}_{k}|\p\Omega=v^{\pm}_{k}|\p\Omega=\varphi_k^\pm$ and
\begin{align*}
&a^{ij}(x,\n u^{\pm}_{k})(u^{\pm}_{k})_{i;j}-\langle\nb f,\nu^{\pm}_{k}\rangle=0\\
&a^{ij}(x,\n v^{\pm}_{k})(v^{\pm}_{k})_{i;j}+F^{\pm}=0
\end{align*}
in $\Omega$, where $a^{ij}$ is as in \eqref{aijdef} and $\nu^{\pm}_{k}$ is the downward unit normal to the graph of $u^{\pm}_{k}$.
Since 
\begin{align*}
a^{ij}(x,\n v^{+}_{k})(v^{+}_{k})_{i;j} + F^{-} & \le a^{ij}(x,\n v^{+}_{k})(v^{+}_{k})_{i;j} + F^{+} = 0 \\
&=
a^{ij}(x,\n v^{-}_{\ell})(v^{-}_{\ell})_{i;j} + F^{-}
\end{align*}
and $v^{+}_{k}|\p\Omega \ge v^{-}_{\ell}|\p\Omega$ for all $k,\ell$, we obtain from the comparison principle \cite[Theorem 10.1]{GilTru} applied to the 
operator $a^{ij}+F^{-}$ that 
\[
v^{-}_{\ell}\le v^{+}_{k}\quad\text{in }\bar{\Omega}.
\]
On the other hand, since $\varphi^{+}_{k+1}\le \varphi^{+}_{k}$ and $\varphi^{-}_{\ell}\le\varphi^{-}_{\ell+1}$ on $\p\Omega$, we have again by the comparison principle that
\begin{equation}\label{v_kOrder}
v_1^- \le \cdots \le v_{\ell}^- \le v_{\ell+1}^- \cdots\le  v_{k+1}^+ \le  v_{k}^+ \cdots \le v_1^+.
\end{equation}      
Similarly, since
\begin{align*}
a^{ij}(x,\n v^{+}_{k})(v^{+}_{k})_{i;j} - \langle \nb f,\nu^{+}_k\rangle &\le
a^{ij}(x,\n v^{+}_{k})(v^{+}_{k})_{i;j} - F^{-} = 0 \\
&=
a^{ij}(x,\n u^{+}_{k})(u^{+}_{k})_{i;j} - \langle \nb f,\nu^{+}_k\rangle
\end{align*}
and $v^{+}_k|\p\Omega =u^{+}_k|\p\Omega$, we get
\[
u^{+}_{k}\le v^{+}_{k}\quad\text{in }\bar{\Omega}.
\]
Similar reasoning implies that $v^{-}_k\le u^{-}_k$, and therefore
\begin{equation}\label{v_ku_kOrder}
v^{-}_{k}\le u^{\pm}_{k}\le v^{+}_k\quad\text{in }\bar{\Omega}.
\end{equation}
Hence the sequences $u_k^\pm,v_k^\pm$ have uniformly bounded $C^0$ norms and
the local interior gradient estimate (Lemma~\ref{globestim}) together with \cite[Corollary 6.3]{GilTru} imply that
the sequences $u_k^\pm,v_k^\pm$ have equicontinuous $C^{2,\alpha}$ 
norms on compact subsets $K\subset \Omega$. 
Taking an exhaustion of $\Omega$ by
compact sets we obtain, with a diagonal argument, that $u_k^\pm$ and $v_k^\pm$
contain subsequences that converge uniformly in compact subsets to functions 
$u,v^\pm \in C^2(\Omega)$ with respect to the $C^2$ norm. Moreover, we have
        \[
                a^{ij}(x,\nabla u)u_{i;j} - \ang{\nb f,\nu}=0 
                \quad \text{and} \quad a^{ij}(x,\nabla v^{\pm})v^\pm_{i;j} + F^\pm =0.
        \]
Since $v_k^\pm |\p\Omega = \varphi_k^\pm$ convergences to $\varphi$, \eqref{v_kOrder}
implies that $v^\pm$ extends continuously to the boundary $\p\Omega$ and $v^\pm |\p\Omega=\varphi$.
In turn, this and \eqref{v_ku_kOrder} give that $u$ extends continuously to $\p\Omega$ with
$u|\p\Omega = \varphi$. Furthermore, because $f\in C^2 (M\times\R)$, it follows that $u\in C^{2,\alpha}(\Omega)
\cap C(\bar{\Omega})$ (\cite[Theorem 6.17]{GilTru}).
\end{proof}

\section{Dirichlet problem at infinity}\label{DPat8}

In this section we assume that $M$ is a Cartan-Hadamard manifold of dimension $n\ge 2$, $\partial_{\infty}M$ is the asymptotic boundary of $M$, 
and $\bar{M}=M\cup\partial_{\infty}M$ the compactification of $M$ in the cone topology.
Recall that the asymptotic boundary is defined as the set of all equivalence classes of unit speed
geodesic rays in $M$; two such rays $\gamma_{1}$ and $\gamma_{2}$ are
equivalent if $\sup_{t\ge0}d\bigl(\gamma_{1}(t),\gamma_{2}(t)\bigr)< \infty$. The equivalence class of $\gamma$ is denoted by $\gamma(\infty)$.
For each $x\in M$ and $y\in\bar{M}\setminus\{x\}$ there exists a unique unit
speed geodesic $\gamma^{x,y}\colon\mathbb{R}\to M$ such that $\gamma^{x,y}_{0}=x$ 
and $\gamma^{x,y}_{t}=y$ for some $t\in(0,\infty]$. If $v\in
T_{x}M\setminus\{0\}$, $\alpha>0$, and $r>0$, we define a cone
\[
C(v,\alpha)=\{y\in\bar M\setminus\{x\}:\sphericalangle(v,\dot\gamma^{x,y}_{0})<\alpha\}
\]
and a truncated cone
\[
T(v,\alpha,r)=C(v,\alpha)\setminus\bar B(x,r),
\]
where $\sphericalangle(v,\dot\gamma^{x,y}_{0})$ is the angle between vectors
$v$ and $\dot\gamma^{x,y}_{0}$ in $T_{x} M$. All cones and open balls in $M$
form a basis for the cone topology on $\bar M$.

Throughout this section, we assume that the sectional curvatures of $M$ are bounded from below and above
by
    \begin{equation}
    \label{curv-bound-gen}
     -(b\circ \rho)^2(x) \le K(P_x) \le -(a\circ \rho)^2 (x)
    \end{equation}
for all $x\in M$,  
where $\rho(x) = d(o,x)$ is the distance to a fixed point $o\in M$ and $P_x$ is any 2-dimensional subspace of $T_xM$. The functions
$a,b\colon [0,\infty) \to [0,\infty)$ are assumed to be smooth such that $a(t)=0$ and $b(t)$ is constant for $t\in [0,T_0]$ for some $T_0>0$, 
and $b\ge a$. Furthermore, we assume that $b$ is monotonic and that there exist positive constants 
$T_1, C_1, C_2, C_3$, and $Q\in (0,1)$ such that 
  \begin{align}
    \tag{A1}\label{A1}
    a(t)\begin{cases}=C_1t^{-1}&\text{if $b$ is decreasing,}\\ 
    \ge C_1t^{-1}&\text{if $b$ is increasing}\\ \end{cases} 
  \end{align}
for all $t\ge T_1$ and
  \begin{align}
    \tag{A2}\label{A2}
    a(t)&\le C_2, \\
    \tag{A3}\label{A3}
    b(t+1)&\le C_2b(t), \\
    \tag{A4}\label{A4}
    b(t/2)&\le C_2b(t), \\
    \tag{A5}\label{A5}
    b(t)&\ge C_3(1+t)^{-Q}
  \end{align}
for all $t\ge 0$.
In addition, we assume that 
  \begin{align}
    \tag{A6}\label{A6}
    &\lim_{t\to\infty}\frac{b'(t)}{b(t)^2}=0 
  \end{align}
and that there exists a constant $C_4>0$ such that
  \begin{align}
    \tag{A7}\label{A7}
    &\lim_{t\to\infty}\frac{t^{1+C_4}b(t)}{f_a'(t)}=0.
  \end{align}
It can be checked from \cite{HoVa} or from \cite{casterasholopainenripoll1} that the curvature
bounds in Corollary \ref{thm1} and Corollary \ref{HVkor2_RT} satisfy the assumptions \eqref{A1}-\eqref{A7}.

\subsection{Construction of a barrier}\label{subsec_barrier_constr}
The curvature bounds \eqref{curv-bound-gen} are needed to control the first two derivatives of the ``barrier'' functions that we will construct in this subsection. Recall from the introduction that for a smooth function 
$k\colon [0,\infty)\to [0,\infty)$, we denote by $f_k \colon [0,\infty) \to \R$
the smooth non-negative solution to the initial value problem
    \begin{equation*}
      \left\{
	\begin{aligned}
	  f_k(0)&=0, \\
	  f_k'(0)&=1, \\
	  f_k''&=k^2f_k.
	\end{aligned}
      \right.
    \end{equation*}
Following \cite{HoVa}, we construct a barrier function for each boundary point $x_0\in\partial_{\infty}M$. 
Towards this end let $v_{0}=\dot\gamma^{o,x_{0}}_{0}$ be the initial (unit) vector of
the geodesic ray $\gamma^{o,x_{0}}$ from a fixed point $o\in M$ and define a
function $h:\partial_{\infty}M\to\mathbb{R}$,
\begin{equation}
\label{eq:hoodef}h(x)=\min\bigl(1,L\sphericalangle(v_{0},\dot\gamma^{o,x}_{0})\bigr),
\end{equation}
where $L\in(8/\pi,\infty)$ is a constant. Then we define a crude extension 
$\tilde h\in C(\bar{M})$, with $\tilde h|\partial_{\infty}M=h$, by setting
\begin{equation}
\label{eq:hoodeftilde}\tilde h(x)=\min\Bigl(1,\max\bigl(2-2\rho
(x),L\sphericalangle(v_{0},\dot\gamma^{o,x}_{0})\bigr)\Bigr).
\end{equation}
Finally, we smooth out $\tilde{h}$ to get an extension
$h\in C^{\infty}(M)\cap C(\bar{M})$ with controlled first and second order 
derivatives. For that purpose, we fix $\chi\in C^{\infty}(\R)$ such that 
$0\le\chi\le 1$, $\spt\chi\subset[-2,2]$, and
$\chi\vert[-1,1]\equiv1$. Then for any function $\varphi\in C(M)$ we define
functions $F_{\varphi}\colon M\times M\to\mathbb{R},\ {\mathcal{R}}
(\varphi)\colon M\to M$, and ${\mathcal{P}}(\varphi)\colon M\to\mathbb{R}$ by
\begin{align*}
F_{\varphi}(x,y)  &  =\chi\bigl(b(\rho(y))d(x,y)\bigr)\varphi(y),\\
{\mathcal{R}}(\varphi)(x)  &  =\int_{M}F_{\varphi}(x,y) dm(y),\ \text{ and}\\
{\mathcal{P}}(\varphi)  &  =\frac{{\mathcal{R}}(\varphi)}{{\mathcal{R}}(1)},
\end{align*}
where
\[
{\mathcal{R}}(1)(x)=\int_{M}\chi\bigl(b(\rho(y))d(x,y)\bigr)dm(y)>0.
\]
If $\varphi\in C(\bar M)$, we extend ${\mathcal{P}}(\varphi)\colon
M\to\mathbb{R}$ to a function $\bar{M}\to\mathbb{R}$ by setting 
${\mathcal{P}}(\varphi)(x)=\varphi(x)$ whenever $x\in M(\infty)$. Then the extended
function ${\mathcal{P}}(\varphi)$ is $C^{\infty}$-smooth in $M$ and continuous
in $\bar{M}$; see \cite[Lemma 3.13]{HoVa}. In particular, applying
${\mathcal{P}}$ to the function $\tilde{h}$ yields an appropriate smooth
extension
\begin{equation}
\label{extend_h}
h:={\mathcal{P}}(\tilde{h})
\end{equation}
of the original function $h\in C\bigl(\partial_{\infty}M\bigr)$ that was defined in \eqref{eq:hoodef}.

We denote
\[
\Omega=C(v_{0},1/L)\cap M \ \text{ and }\ \ell\Omega=C(v_{0},\ell/L)\cap M
\]
for $\ell>0$ and collect various constants and functions together to a data
\[
C=(a,b,T_{1},C_{1},C_{2},C_{3},C_{4},Q,n,L).
\]
Furthermore, we denote by $\|\Hess_{x} u\|$ the norm of the Hessian of a
smooth function $u$ at $x$, that is
\[
\|\Hess_{x} u\|=\sup_{ \overset{ \mbox{\scriptsize$X\in T_xM$} }{\lvert X
\rvert\le1}}\lvert\Hess u(X,X) \rvert.
\]
The following lemma gives the desired estimates for derivatives of $h$.
We refer to \cite{HoVa} for the proofs of these estimates; see also \cite{CHR3}.

\begin{lem}\cite[Lemma 3.16]{HoVa}\label{arvio_lause}
There exist constants $R_1=R_1(C)$ and $c_1=c_1(C)$ such that the extended function 
$h\in C^\infty(M)\cap C(\bar M)$ in \eqref{extend_h}
satisfies 
\begin{equation}\label{arvio1}
  \begin{split}
  |\nabla h(x)|&\le c_1\frac{1}{(f_a\circ\rho)(x)}, \\
  \|\Hess_x h\|&\le c_1\frac{(b\circ\rho)(x)}{(f_a\circ\rho)(x)}, \\
  \end{split}
  \end{equation}
for all $x\in 3\Omega\setminus B(o,R_1)$. 
In addition,
\[h(x)=1
  \]
for every $x\in M\setminus\bigl(2\Omega\cup B(o,R_1)\bigr)$.
\end{lem}

We define a function $F\colon M\to [0,\infty)$ and an elliptic operator $\tilde{Q}$ by setting
\begin{equation}\label{defF}
F(x)=\sup_{t\in\R}\abs{\bar{\nabla}f(x,t)}
\end{equation}
and
\begin{equation}\label{defqtilde}
\tilde{Q}[v]=\dv\frac{\nabla v}{\sqrt{1+\abs{\n v}^2}}+F(x).
\end{equation}
Let then $A>0$ be a fixed constant. We aim to show that 
\begin{equation}\label{varphi_def}
\psi=A(R_{3}^{\delta}\rho^{-\delta}+h)
\end{equation}
is a supersolution $\tilde{Q}[\psi]< 0$
in the set $3\Omega\setminus\bar{B}(o,R_{3})$, where $\delta>0$ and $R_{3}>0$ are constants that will be specified later and $h$ is the extended function defined in \eqref{extend_h}. 
We shall make use of the following estimates obtained in \cite{HoVa}:
\begin{lem}\cite[Lemma 3.17]{HoVa}\label{perusta}
There exist constants $R_2=R_2(C)$ and $c_2=c_2(C)$ with the following property.
If $\delta\in(0,1)$, then
\[\begin{split}
  |\nabla h|&\le c_2/(f_a\circ\rho), \\
  \|\Hess h\|&\le c_2\rho^{-C_4-1}(f_a'\circ\rho)/(f_a\circ\rho), \\
  |\nabla\langle\nabla h,\nabla h\rangle|&\le c_2\rho^{-C_4-2}(f_a'\circ\rho)/(f_a\circ\rho), \\
  |\nabla\langle\nabla h,\nabla(\rho^{-\delta})\rangle|&\le c_2\rho^{-C_4-2}(f_a'\circ\rho)/(f_a\circ\rho), \\
  \nabla\bigl\langle\nabla(\rho^{-\delta}),\nabla(\rho^{-\delta})\bigr\rangle
  &=-2\delta^2(\delta+1)\rho^{-2\delta-3}\nabla\rho
  \end{split}\]
in the set $3\Omega\setminus B(o,R_2)$.
\end{lem}

As in \cite{HoVa} we denote
\[
\phi_{1}=\frac{1+\sqrt{1+4C_{1}^{2}}}{2}>1,\quad\text{and}\quad\delta_{1}
=\min\left\lbrace C_{4},\frac{-1+(n-1)\phi_{1}}{1+(n-1)\phi_{1}}\right\rbrace
\in(0,1),
\]
where $C_{1}$ and $C_{4}$ are constants defined in \eqref{A1} and \eqref{A7},
respectively. 
\begin{lem}\label{PsiBarrierLemma}
 Let $A>0$ be a fixed constant and $h$ the 
function defined in \eqref{extend_h}. Assume that the function $F$ defined in \eqref{defF} satisfies
\begin{equation}\label{Fcond}
\sup_{\rho(x)=t}F(x)=o\left(\frac{f^{\prime}_a (t)}{f_a(t)}t^{-\ve-1}\right)
\end{equation} 
for some $\epsilon>0$ as $t\to\infty$. Then there 
 exist two positive constants $\delta\in(0,\min(\delta_1,\ve))$ and $R_3$ 
depending on $C$ and $\varepsilon$ such that the function 
 $\psi=A(R_3^{\delta} \rho^{-\delta} + h)$ satisfies $\tilde{Q}[\psi]< 0$
 in the set $3\Omega\setminus \bar{B}(o,R_3)$.
\end{lem}
\begin{proof}
 In the proof $c$ will denote a positive constant whose actual value may vary even within a line.
 Since
    \begin{align*}
     \tilde{Q}[\psi] &= \frac{\Delta \psi}{\sqrt{1+|\nabla\psi|^2}} - \frac{1}{2} \frac{\ang{\nabla|\nabla\psi|^2,
	\nabla\psi}}{(1+|\nabla\psi|^2)^{3/2}} +F(x) \\
	&= \frac{(1+|\nabla \psi|^2) \Delta\psi + (1+|\nabla\psi|^2)^{3/2} F(x) - \frac{1}{2}
	\ang{\nabla|\nabla\psi|^2,\nabla\psi}}{(1+|\nabla\psi|^2)^{3/2}},
    \end{align*}
it is enough to show that there exist $\delta >0$ and $R_3$
such that
    \begin{equation} \label{PsiBarrierIneq}
     (1+|\nabla \psi|^2) \Delta\psi + (1+|\nabla\psi|^2)^{3/2} F(x) - \frac{1}{2}
	\ang{\nabla|\nabla\psi|^2,\nabla\psi} < 0
    \end{equation}
in the set $3\Omega \setminus \bar{B}(o,R_3)$.
    
First we notice that $\psi$ is $C^{\infty}$-smooth and
\[
\nabla\psi=A\bigl(-R_3^{\delta}\delta\rho^{-\delta-1}\nabla\rho +\nabla h\bigr)  
\]
in $M\setminus\{o\}$.
Lemma \ref{perusta} and our curvature assumption imply that 
$|\nabla h| \le c/\rho$ for $\rho$ large enough, and therefore
    \begin{equation*}
      |\nabla \psi|^2 = (AR_3^\delta)^2 \delta^2 \rho^{-2\delta-2} + A^2|\nabla h|^2 - 2A^2R_3^\delta \delta 
      \rho^{-\delta-1} \ang{\nabla \rho,\nabla h}
      \le c\rho^{-2}
    \end{equation*}
in $3\Omega \setminus \bar{B}(o,R_3)$ for sufficiently large $R_3$. Then, to estimate the term with $\Delta \psi$ in \eqref{PsiBarrierIneq}, we first note that
    \[
      \Delta \psi = AR_3^\delta \big( \delta(\delta+1)\rho^{-\delta-2} -\delta \rho^{-\delta-1}\Delta \rho \big) 
      + A\Delta h.
    \]
Furthermore, for every $\delta\in (0,\delta_1)$, there exists $R_3=R_3(C,\delta)$ such that
    \[
      \Delta \rho \ge (n-1) \frac{f_a' \circ \rho}{f_a\circ \rho}
\ge \frac{(n-1)(1-\delta)\phi_1}{\rho}      
       >0
    \]
whenever $\rho\ge R_3$; see \cite[(3.25)]{HoVa}.
Therefore, using Lemma \ref{perusta}, we obtain
    \begin{align*}
(1+|\nabla \psi|^2)\Delta\psi &\le (1+|\nabla \psi|^2)AR_3^\delta \delta\left( \delta+1 
     - (n-1) \frac{\rho f_a' \circ \rho}{f_a\circ \rho} \right)\rho^{-\delta-2}  \\
&\qquad + (1+|\nabla \psi|^2) Anc_2 \left(\frac{f_a' \circ \rho}{f_a\circ \rho} \right)   
      \rho^{-C_4-1} \\
&\le AR_3^\delta \delta \left( \delta+1  -(n-1) \frac{\rho f_a' \circ \rho}
     {f_a\circ \rho} \right) \rho^{-\delta-2}     \\ 
&\qquad + \left(1+ c\rho^{-2}\right) Anc_2 \left(\frac{\rho f_a' \circ \rho}{f_a\circ \rho}\right)\rho^{-C_4 -2}\\
&= -\left(\frac{\rho f_a' \circ \rho}{f_a\circ \rho}\right)\rho^{-\delta-2}
     \left(AR_3^{\delta}\delta(n-1)-(1+c\rho^{-2})Anc_2\rho^{\delta-C_4}\right)\\
     &\qquad + AR_3^{\delta}\delta(\delta+1)\rho^{-\delta-2}\\
     &\le -c\left(\frac{\rho f_a^{\prime}\circ \rho}{f_a \circ\rho}\right)\rho^{-\delta-2}
    \end{align*}
whenever $\delta\in (0,\delta_1)$ is small enough and $\rho\ge R_3(C,\delta)$.     
These estimates hold since
\[
\delta +1 -(n-1)\frac{\rho f_a^{\prime}\circ\rho}{f_a\circ\rho}\le \delta +1 -(n-1)(1-\delta)\phi_1\le 0
\]
for a sufficiently small $\delta\in (0,\delta_1)$.
 Now taking into account our assumption \eqref{Fcond} we obtain
 \begin{align}
 \begin{split}
  (1+\abs{\nabla\psi}^2) \Delta\psi + (1+\abs{\nabla\psi}^2)^{3/2}F &
 \le  -c\left(\frac{\rho f_a^{\prime}\circ \rho}{f_a \circ\rho}\right)\rho^{-\delta-2} +
 (1+c\rho^{-2})F \label{two-terms-est}\\
 &\le  -c\left(\frac{\rho f_a^{\prime}\circ \rho}{f_a \circ\rho}\right)\rho^{-\delta-2}
\end{split} 
  \end{align}
whenever $\delta\in (0,\min(\varepsilon,\delta_1))$ is small enough and $\rho\ge R_3(C,\delta)$.

It remains to estimate $\abs{\ang{\nabla|\nabla\psi|^2,\nabla\psi}}$ from above. Since
\[
\nabla\psi=AR_{3}^{\delta}\nabla(\rho^{-\delta})+A\nabla h,
\]
we have
\begin{align*}
\nabla|\nabla\psi|^2 &= A^2\nabla\ang{R_3^{\delta}\nabla(\rho^{-\delta})+\nabla h,R_3^{\delta}\nabla(\rho^{-\delta})+\nabla h}\\
&=(AR_3^{\delta})^2\nabla\ang{\nabla(\rho^{-\delta}),\nabla(\rho^{-\delta})} 
+2A^2 R_3^{\delta}\nabla\ang{\nabla(\rho^{-\delta}),\nabla h}
+A^2 \nabla\ang{\nabla h,\nabla h}.
\end{align*}
By Lemma \ref{perusta} we then get
\begin{align}\label{3rd-term-est}
\abs{\ang{\nabla|\nabla\psi|^2,\nabla\psi}} &\le c\rho^{-1}\left(2(\delta AR_3^{\delta})^2(\delta+1)\rho^{-2\delta-3}
+A^2c_2(2R_3^\delta +1)\left(\frac{f_a^\prime \circ\rho}{f_a \circ\rho}\right)\rho^{-C_4 -2}\right)\nonumber\\
&\le c\delta^2 (\delta+1)\rho^{-2\delta-4} + c\left(\frac{\rho f_a^{\prime}\circ\rho}{f_a\circ\rho}\right)\rho^{-C_4 -4}\\
&\le c\left(\rho^{-2\delta-4}+\rho^{-C_4 -4}\right)\frac{\rho f_a^{\prime}\circ\rho}{f_a\circ\rho}.\nonumber
\end{align}
Putting together \eqref{two-terms-est} and \eqref{3rd-term-est} we finally obtain
\[ 
 (1+|\nabla \psi|^2) \Delta\psi + (1+|\nabla\psi|^2)^{3/2} F(x) - \frac{1}{2}
	\ang{\nabla|\nabla\psi|^2,\nabla\psi}
	\le 
-c\left(\frac{\rho f_a^{\prime}\circ \rho}{f_a \circ\rho}\right)\rho^{-\delta-2}<0	
\]
in $3\Omega \setminus \bar{B}(o,R_3)$ for a sufficiently small $\delta>0$ and large $R_3$.
\end{proof}
Similarly, we have
\begin{equation}
\label{subsol}
\dv\frac{\nabla(-\psi)}{\sqrt{1+\abs{\nabla(-\psi)}^2}} - F(x)>0
\end{equation}
in $3\Omega\setminus \bar{B}(o,R_3)$.

\subsection{Uniform height estimate} 
We will solve the asymptotic Dirichlet problem by solving the problem first in a sequence of
balls with increasing radii. In order to obtain a converging subsequence of solutions, we need
to have a uniform height estimate. This subsection is devoted to the construction of a barrier function
that will guarantee the height estimate.

Since $f_a''-a^2f_a=0$, where $a(t)=0$ for $t\in [0,T_0]$ and 
\[
a(t)\ge \frac{\sqrt{\phi(\phi-1)}}{t}
\] 
for $t\ge T_1$ and some $\phi>1$, we have $f_a(t)\ge ct^\phi$ for $t\ge T_1$. Therefore
    \begin{equation}\label{int_psi_finite}
      \int_1^\infty \frac{dr}{f_a^{n-1}(r)} < \infty.
    \end{equation}
Let $\varphi \colon M \to \R$ be a bounded function. We aim to show the existence of a barrier function
$V$ such that $\tilde{Q}[V] \le 0$ and $V(x)>||\varphi||_\infty$ in $M$.  
In order to define such a function $V$, we need an auxiliary function $a_0>0$, so that 
    \begin{equation}\label{HP1}
      \int_1^\infty \left( \int_r^\infty \frac{ds}{f_a^{n-1}(s)} \right) a_0(r)f_a^{n-1}(r) dr <\infty.
    \end{equation}
We will discuss about the choice of $a_0$ in Examples \ref{a_0_exam1} and \ref{a_0_exam2}.
Now, following \cite{mastrolia2015elliptic}, we can define 
    \begin{equation} \label{Vdefinition}\begin{split}
      V(x) = V\big(\rho(x)\big) &= \left(\int_{\rho(x)}^\infty \frac{ds}{f_a^{n-1}(s)}\right) \left(\int_0^{\rho(x)} a_0(t)
      f_a^{n-1}(t) dt \right) \\ &\quad - \int_0^{\rho(x)} \left( \int_t^\infty \frac{ds}{f_a^{n-1}(s)}\right) a_0(t) 
      f_a^{n-1}(t) dt - H + ||\varphi||_\infty,\end{split}
    \end{equation}
 where 
      \begin{equation}\label{Hdefinition}\begin{split}
	H\coloneqq \limsup_{r\to\infty} \Big\{ \int_{r}^\infty &\frac{ds}{f_a^{n-1}(s)} \int_0^{r} a_0(t)
      f_a^{n-1}(t) dt  \\ &- \int_0^{r}  \int_t^\infty \frac{ds}{f_a^{n-1}(s)} a_0(t) 
      f_a^{n-1}(t) dt \Big\} \le 0;\end{split}
      \end{equation}
see  \cite[(4.5)]{mastrolia2015elliptic}.    
From \eqref{int_psi_finite} and \eqref{HP1} we see that $H$ is finite and hence $V$ is well defined.
      
As in the proof of Lemma \ref{PsiBarrierLemma}, we write 
\begin{equation}\label{Qfraction}
     \tilde{Q}[V] = \frac{(1+|\nabla V|^2) \Delta V + (1+|\nabla V|^2)^{3/2}F(x) - \frac{1}{2}
  \ang{\nabla |\nabla V|^2, \nabla V}}{(1+|\nabla V|^2)^{3/2}},
   \end{equation}
where $F(x)$ is as in \eqref{defF}, and estimate the terms of the numerator. To begin, we notice that
    \begin{align*}
      V'(r) &= - \frac{1}{f_a^{n-1}(r)} \int_0^r a_0(t) f_a^{n-1}(t) dt < 0,\\
      V''(r)&= (n-1) \frac{f_a^\prime(r)}{f_a^n(r)} \int_0^r a_0(t) f_a^{n-1}(t) dt - a_0(r),\\
\noalign{and}  
      \left|\nabla V\big(\rho(x)\big)\right| &= \left| V'\big(\rho(x)\big) \nabla \rho(x)\right| 
      = \left|V'\big(\rho(x)\big)\right|. 
       \end{align*}
Note that $-V(r)=g(r)$, the function \eqref{gV} in Introduction.
The Laplace comparison theorem implies that
    \[
      \Delta \rho \ge (n-1)\frac{f_a^\prime\circ  \rho}{f_a \circ \rho}.
    \]
Hence we can estimate the Laplacian of $V$ as
\begin{align*}
      \Delta V &= V''\big(\rho\big) + \Delta \rho V'\big(\rho\big) \\ 
      &\le V''(\rho) + (n-1)\frac{f_a^\prime(\rho)}{f_a(\rho)}V'(\rho) \\
      &= (n-1) \frac{f_a^\prime(\rho)}{f_a^n(\rho)} \int_0^\rho a_0(t) f_a^{n-1}(t) dt - a_0(\rho) 
       - (n-1)\frac{f_a^\prime(\rho)}{f_a^n(\rho)}
      \int_0^\rho a_0(t) f_a^{n-1}(t) dt \\
      &= -a_0(\rho), 
    \end{align*}
  and thus the first term of \eqref{Qfraction} can be estimated as
\[ 
     \big(1 + |\nabla V|^2\big) \Delta V \le -\big(1 + |\nabla V|^2\big)a_0(\rho)
     \le -\big(1+V'(\rho)^2\big) a_0(\rho).  
\]     
Then, for the last term of \eqref{Qfraction} we have 
 \begin{align*}
    -\frac{1}{2}\ang{\nabla|\nabla V|^2,\nabla V} &= -\frac{1}{2}\ang{\nabla (V'(\rho))^2, V'(\rho) \nabla \rho} = -\frac{1}{2}\ang{2V'(\rho) V''(\rho) \nabla \rho,
  V'(\rho)\nabla \rho} \\
  &= -\big(V'(\rho)\big)^2 V''(\rho) \\
  &= \frac{-1}{f_a^{2n-2}(\rho)} \left(\int_0^\rho a_0(t) f_a^{n-1}(t) dt \right)^2 \\
      &\qquad \cdot \left((n-1) \frac{f_a^\prime(\rho)}{f_a^n(\rho)}
	  \int_0^\rho a_0(t) f_a^{n-1}(t) dt - a_0(\rho) \right) \\
  &= \frac{a_0(\rho)}{f_a^{2n-2}(\rho)} \left(\int_0^\rho a_0(t) f_a^{n-1}(t) dt \right)^2 \\
     &\qquad  -\frac{(n-1)f_a^\prime(\rho)}{f_a^{3n-2}(\rho)} \left(\int_0^\rho a_0(t) f_a^{n-1}(t) dt \right)^3\\
   &= a_0(\rho) V'(\rho)^2 - (n-1) \frac{f_a^\prime(\rho)}{f_a(\rho)} \big( -V'(\rho) \big)^3.
   \end{align*}
 Collecting everything together, we obtain that $\tilde{Q}[V] \le 0$ if
   \begin{align*}
      \sup_{\p B(o,r) \times\R} |\bar\nabla f| \le  \frac{a_0(r) + (n-1) 
	\frac{f_a^\prime(r)}{f_a(r)}\big( -V'(r) \big)^3}{\big( 1+V'(r)^2 \big)^{3/2}}.
    \end{align*}

Finally it is easy to check that, 
since $H$ is finite and $V$ is decreasing, we have
$V(x)>||\varphi||_\infty$ for all $x\in M$ and $V(x) \to ||\varphi||_\infty$ as $\rho(x) \to \infty$.
Altogether, we have obtained the following.

\begin{lem}\label{UniformBoundLem}
 Let $\varphi \colon M \to \R$ be a bounded function and assume that the function $V$ defined in \eqref{Vdefinition} satisfies 
     \begin{align}\label{grad_fCond}
       \sup_{\p B(o,r) \times\R} |\bar\nabla f| \le \frac{a_0(r) + (n-1) 
	\frac{f_a^\prime(r)}{f_a(r)}\big( -V'(r) \big)^3}{\big( 1+V'(r)^2 \big)^{3/2}}.
    \end{align}
Then the function $V$ is an upper barrier for the Dirichlet problem such that
\begin{equation}\label{Vsuper}
       \tilde{Q}[V] =\dv\frac{\nabla V}{\sqrt{1+\abs{\nabla V}^2}} + F(x)\le 0 \quad \text{in } M,
      \end{equation}
      \begin{equation}\label{Vheight}
       V(x) > ||\varphi||_\infty \quad \text{for all } x\in M 
      \end{equation}
and
      \begin{equation}\label{Vlimes}
       \lim_{r(x)\to\infty} V(x) = ||\varphi||_\infty.
      \end{equation}
Furthermore,
\begin{equation}
\label{-Vsubsol}
\dv\frac{\nabla(-V)}{\sqrt{1+\abs{\nabla(-V)}^2}} - F(x)\ge 0\quad \text{in } M.
\end{equation}
\end{lem}
 
 Next we show by examples that in the situation of Corollaries \ref{thm1} and \ref{HVkor2_RT} the 
 condition \eqref{grad_fCond} is not a stronger restriction than the assumption \eqref{Fcond} in Lemma
 \ref{PsiBarrierLemma}. First note that $V'(r) \to 0$ as $r\to\infty$, and hence the upper 
 bound \eqref{grad_fCond} for $|\nb f|$ is asymptotically the function $a_0$.
 
 \begin{exa}\label{a_0_exam1}
  Assume that the sectional curvatures of $M$ satisfy 
  \[
   K(P_x)\le -a\big(\rho(x)\big)^2 =- \frac{\phi(\phi-1)}{\rho(x)^2}, \quad \phi>1,
  \]
for $\rho(x)\ge T_1$. 
We need to choose the function $a_0$ such that \eqref{HP1} holds, and since this is a 
question about its asymptotical behaviour, it is enough to consider the integral
    \[
      \int_{T_1}^\infty \left( \int_r^\infty \frac{ds}{f_a^{n-1}(s)} \right) a_0(r)f_a^{n-1}(r) dr.
    \]
For $t\ge T_1$, $f_a(t) = c_1 t^\phi + c_2 t^{1-\phi}$,  and hence, by a straightforward computation, we have \eqref{HP1} if 
    \[
     \int_{T_1}^\infty a_0(r)r \, dr < \infty.
    \]
So it is enough to choose for example
    \[
     a_0(r) = O\left(\frac{1}{r^2(\log r)^{\alpha}}\right)
    \]
as $r\to\infty$ for some $\alpha>1$. On the other hand, with this curvature upper bound, the assumption \eqref{Fcond} 
requires decreasing of order $o\bigl(r^{-2-\ve}\bigr)$.
 \end{exa}

 \begin{exa}\label{a_0_exam2}
  Assume that the sectional curvatures of $M$ satisfy 
  \[
    K\le-k^2, 
  \]
for $\rho(x)\ge T_1$ and some constant $k>0$. Then, for large $t$,
$f_a(t) = c_1 \sinh kt+ c_2\cosh kt\approx e^{kt}$. Therefore
it is straightforward to see that we have \eqref{HP1} if
    \[
      \int_{T_1}^\infty a_0(r)\,dr <\infty,
    \]
which holds by choosing, for example, 
    \[
      a_0(r) = O\left(\frac{1}{r(\log r)^\alpha}\right),\ \alpha>1,
    \]
as $r\to\infty$. On the other hand, with this curvature upper bound, the assumption \eqref{Fcond}
requires decreasing of order $o\bigl(r^{-1-\ve}\bigr)$.
 \end{exa}

\subsection{Proof of Theorem \ref{ThmMain}}
We start with solving the Dirichlet problem in geodesic balls $B(o,R)$.
\begin{lem}\label{adp-gb} 
Suppose that $f\in C^2(M\times\R)$ is of the form $f(x,t)=m(x)+r(t)$ and 
satisfies
\[
\sup_{\partial B(o,r)\times\R}\abs{\bar{\nabla}f} \le (n-1)\frac{f_a^\prime (r)}{f_a(r)}
\]
for all $r>0$. Then for every $R>0$ and $\varphi\in C(\partial B(o,R))$ there exists a solution
$u\in C^{2,\alpha}(B(o,R))\cap C(\bar{B}(o,R))$ of the Dirichlet problem
\begin{equation*} 
 \begin{cases}
  \dv \dfrac{\nabla u}{\sqrt{1+|\nabla u|^2}} = \ang{\nb f,\nu} \quad \text{in } B(o,R) \\
  u|\p B(o,R) = \varphi.
 \end{cases}
\end{equation*}
\end{lem}
\begin{proof}
Assuming first that $\varphi\in C^{2,\alpha}(\partial B(o,R))$ the claim follows by the Leray-Schauder method. Indeed, for each 
$x\in \bar{B}(o,R)\setminus\{o\}$ the inward mean curvature $H(x)$ of the level set
$\{y\in \bar{B}(o,R)\colon d(y)=d(x)\}=\partial B(o,\rho(x))$ satisfies 
\[
H(x)=\Delta\rho(x)\ge (n-1)\frac{f_a^\prime \big(\rho(x)\big)}{f_a\big(\rho(x)\big)}
\ge \sup_{\partial B(o,\rho(x))\times\R}\abs{\bar{\nabla}f}.
\] 
In other words, \eqref{Fxtra-again} and \eqref{H-f-again} hold and therefore we can apply 
the Leray-Schauder method as in the proof of Theorem~\ref{existence}. The general case
$\varphi\in C(\partial B(o,R))$ follows by approximation as in the proof of 
Theorem~\ref{cont_existence}.
 \end{proof}

\begin{proof}[Proof of Theorem \ref{ThmMain}]
We extend the boundary data function $\varphi \in C(\p_\infty M)$ to a function $\varphi \in C(\bar M)$.
Let $\Omega_k = B(o,k), \, k\in\N$, be an exhaustion of $M$. 
By Lemma~\ref{adp-gb}, there exist solutions $u_k\in C^{2,\alpha}(\Omega_k)\cap C(\bar\Omega_k)$ to
\begin{equation*}
 \begin{cases}
Q[u_k]=  \dv \dfrac{\nabla u_k}{\sqrt{1+|\nabla u_k|^2}} - \ang{\nb f,\nu_k} \quad \text{in } \Omega_k \\
  u_k|\p\Omega_k = \varphi,
 \end{cases}
\end{equation*}
where $\nu_k$ is the downward pointing unit normal to the graph of $u_k$.
Applying the uniform height estimate, Lemma~\ref{UniformBoundLem}, 
we see that the sequence $(u_k)$ is uniformly bounded and hence the
interior gradient estimate (Lemma \ref{globestim}), together with the diagonal argument, implies that 
there exists a subsequence, still denoted by $u_k$, that converges locally uniformly
with respect to $C^2$-norm to a solution $u$. Therefore we are left to prove that $u$ extends continuously
to $\pinf M$ and satisfies $u|\pinf M = \varphi$.

Towards that end let us fix $x_0\in\pinf M$ and $\ve>0$. Since the boundary data function $\varphi$ is continuous, we find
$L\in(8/\pi,\infty)$ such that 
    \[
      |\varphi(y) - \varphi(x_0)| < \ve/2
    \]
for all $y \in C(v_0,4/L)\cap \pinf M$, where $v_0=\dot\gamma_0^{o,x_0}$ is the initial vector of the geodesic
ray representing $x_0$. Moreover, by \eqref{Vlimes} we can choose $R_3$ in Lemma~\ref{PsiBarrierLemma} so large that 
$V(r)\le \max_{\bar{M}} |\varphi| + \ve/2$ for $r\ge R_3$.

We claim that
    \begin{equation}\label{squeeze_ineq}
      w^-(x) \coloneqq -\psi(x) + \varphi(x_0) - \ve \le u(x) \le w^+(x) \coloneqq \psi(x) + \varphi(x_0) + \ve
    \end{equation}
in the set $U\coloneqq 3\Omega \setminus \bar{B}(o,R_3)$, where $\psi=A(R_3^\delta \rho^{-\delta} + h)$ is the 
supersolution $\tilde{Q}[\psi]<0$ in Lemma \ref{PsiBarrierLemma} and $A= 2\max_{\bar M}|\tilde\varphi|$. 
Recall the notation $\Omega = C(v_0,1/L) \cap M$ and $\ell\Omega = C(v_0,\ell/L) \cap M,\  \ell>0$,
from Subsection \ref{subsec_barrier_constr}. 

The function $\varphi$ is continuous in $\bar{M}$ so there exists $k_0$ such
that $\p\Omega_{k_0} \cap U \neq \emptyset$, and
    \begin{equation}\label{tilde_phi-phi}
      |\varphi(x) - \varphi(x_0)| < \ve/2
    \end{equation}
for all $x \in \p\Omega_k \cap U$ when $k\ge k_0$. Denote $V_k = \Omega_k \cap U$ for $k\ge k_0$. We will conclude that 
\begin{equation}\label{ineq_in_Vk}
	w^- \le u_k \le w^+
    \end{equation}
in $V_k$ by using the comparison principle for the operator $\tilde Q_k$,
\[
\tilde{Q}_k[v]=\dv \dfrac{\nabla v}{\sqrt{1+|\nabla v|^2}} - \ang{\nb f,\nu_k},
\]
where $\nu_k$ is the downward pointing unit normal to the graph of the solution $u_k$. Notice that
    \[
      \p V_k = (\p\Omega_k \cap \bar U) \cup (\p U \cap \bar\Omega_k).
    \]
Let $x \in \p\Omega_k \cap \bar U$ and $k\ge k_0$. Then \eqref{tilde_phi-phi} and $u_k|\p\Omega_k = \varphi|\p\Omega_k$ imply that
    \[
      w^-(x) \le \varphi(x_0) - \ve/2 \le \varphi(x) = u_k(x) \le \varphi(x_0) + \ve/2 \le w^+(x).
    \]
Moreover, by Lemma \ref{arvio_lause}, we have
    \[
      h|M \setminus \big(2\Omega \cup B(o,R_1)\big) = 1
    \]
and $R_3^\delta \rho^{-\delta} = 1$ on $\p B(o,R_3)$, so
    \[
	\psi \ge A = 2\max_{\bar M} |\varphi|
    \]
on $\p U \cap \bar\Omega_k$. 
By Lemma \ref{UniformBoundLem}, $V$ is a supersolution $\tilde Q[V]\le 0$ and hence
\begin{align*}
      \dv \frac{\nabla V}{\sqrt{1+|\nabla V|^2}} - \ang{\nb f,\nu_k} &\le 
      \dv \frac{\nabla V}{\sqrt{1+|\nabla V|^2}} + F(x)\\
      & =  \tilde Q[V] \le 0\\
      & =\dv \frac{\nabla u_k}{\sqrt{1+|\nabla u_k|^2}} - \ang{\nb f,\nu_{k}}.
\end{align*}
Since $V\ge \max_{\bar M} |\varphi|$ on $\p\Omega_k$, the comparison principle yields $u_k|\Omega_k \le V|\Omega_k,$
and by the choice of $R_3$, we have
    \[
	u_k \le \max_{\bar M} |\varphi| + \ve/2
    \]
in $\Omega_k \setminus B(o,R_3)$.

Altogether, it follows that
    \[
	w^+= \psi + \varphi(x_0) + \ve \ge 2\max_{\bar M} |\varphi| + \varphi(x_0) + \ve 
	\ge \max_{\bar M}|\varphi| + \ve \ge u_k
    \] 
on $\p U\cap \bar\Omega_k$, and similarly $u_k \ge w^-$ on $\p U\cap \bar\Omega_k$. Consequently $w^-\le u_k \le w^+ $
on $\p V_k$. 
By Lemma~\ref{PsiBarrierLemma}, $\tilde{Q}[\psi]<0$, and therefore
\begin{align*}
\tilde{Q}_k[w^+]&=\dv \frac{\nabla w^+}{\sqrt{1+|\nabla w^+|^2}} - \ang{\nb f,\nu_{k}}\\
 &= 
      \dv \frac{\nabla\psi}{\sqrt{1+|\nabla \psi|^2}} - \ang{\nb f,\nu_{k}}\\
       &\le \dv \frac{\nabla\psi}{\sqrt{1+|\nabla \psi|^2}}  + F(x)\\
      & =  \tilde Q[\psi] < 0\\
      & =\dv \frac{\nabla u_k}{\sqrt{1+|\nabla u_k|^2}} - \ang{\nb f,\nu_{k}}
\end{align*}
in $U$. By the comparison principle, $u_k\le w^+$ in $U$. Similarly, using \eqref{subsol} we conclude that
\[
\dv \frac{\nabla w^-}{\sqrt{1+|\nabla w^-|^2}} - \ang{\nb f,\nu_{k}} >
\dv \frac{\nabla u_k}{\sqrt{1+|\nabla u_k|^2}} - \ang{\nb f,\nu_{k}}
\]
in $U$. Hence $u_k\ge w^-$ in $U$ and we obtain \eqref{ineq_in_Vk}. This holds for every $k\ge k_0$ and hence \eqref{squeeze_ineq} follows. 
Finally,
    \[
      \limsup_{x\to x_0} |u(x) - \varphi(x_0)| \le \ve
    \]
since $\lim_{x\to x_0} \psi(x) =0$. Because $x_0\in\pinf M$ and $\ve>0$ were arbitrary, this shows that $u$
extends continuously to $C(\bar M)$ and $u|\pinf M = \varphi$. 
\end{proof}


\end{document}